\documentclass[11pt]{article}

\usepackage{amssymb, amsmath, verbatim, amsthm,url, multirow,fullpage,mathtools}
\usepackage{longtable, rotating,makecell,array}
\usepackage[aligntableaux=top]{ytableau}

\usepackage{soul}
\usepackage[colorinlistoftodos,textsize=footnotesize]{todonotes}

\setstcolor{red}

\usepackage[all]{xy}
\xyoption{poly}
\xyoption{arc}

\usetikzlibrary{decorations.pathmorphing,calc}
\usetikzlibrary{intersections}

\definecolor{light-gray}{gray}{0.6}

\tikzstyle{propagator}=[decorate,decoration={snake,amplitude=0.8mm}]
\tikzstyle{smallpropagator}=[decorate,decoration={snake,segment length=3mm,amplitude=0.5mm}]

\tikzstyle{firstdash}=[dashed,line cap=round, dash pattern=on 2pt off 1pt]
\tikzstyle{seconddash}=[dashed,line cap=round, dash pattern=on 0.5pt off 1pt]

\newcommand{\drawWLD}[2]{

\pgfmathsetmacro{\n}{#1}
\pgfmathsetmacro{\radius}{#2}
\pgfmathsetmacro{\angle}{360/\n}
    \foreach \i in {1,2,...,\n} {
      \pgfmathsetmacro{\x}{\angle*\i}
        \draw[-,shorten >=-\radius*0.1 cm,shorten <=-\radius*0.1 cm]  (\x:\radius cm)-- (\x + \angle: \radius cm);
    }

\draw (\angle:\radius) node {$\bullet$};
}

\newcommand{\drawprop}[4]{
\pgfmathsetmacro{\r}{#1}
\pgfmathsetmacro{\bumpr}{#2}
\pgfmathsetmacro{\s}{#3}
\pgfmathsetmacro{\bumps}{#4}
\pgfmathsetmacro{\perturbe}{\angle/\n}

\begin{scope}
\clip (\angle*\r:\radius) -- (\angle + \angle*\r:\radius) -- (\angle*\s:\radius) -- (\angle + \angle*\s:\radius) -- (\angle*\r:\radius);
\draw[propagator] (\angle*\r + \angle/2 + \bumpr*\perturbe:\radius) -- (\angle*\s + \angle/2 + \bumps*\perturbe:\radius);
\end{scope}
}

\newcommand{\modifiedprop}[5]{
\pgfmathsetmacro{\r}{#1}
\pgfmathsetmacro{\bumpr}{#2}
\pgfmathsetmacro{\s}{#3}
\pgfmathsetmacro{\bumps}{#4}
\pgfmathsetmacro{\perturbe}{\angle/\n}

\begin{scope}
\clip (\angle*\r:\radius) -- (\angle + \angle*\r:\radius) -- (\angle*\s:\radius) -- (\angle + \angle*\s:\radius) -- (\angle*\r:\radius);
\draw[#5] (\angle*\r + \angle/2 + \bumpr*\perturbe:\radius) -- (\angle*\s + \angle/2 + \bumps*\perturbe:\radius);
\end{scope}
}

\newcommand{\boundaryprop}[4]{
\pgfmathsetmacro{\r}{#1}
\pgfmathsetmacro{\bumpr}{#2}
\pgfmathsetmacro{\s}{#3}
\pgfmathsetmacro{\perturbe}{\angle/\n}

\begin{scope}
\clip (\angle*\r:\radius) -- (\angle + \angle*\r:\radius) -- (\angle*\s - \angle:\radius) -- (\angle*\s:\radius) -- (\angle + \angle*\s:\radius) -- (\angle*\r:\radius);
\draw[#4] (\angle*\r + \angle/2 + \bumpr*\perturbe:\radius) -- (\angle*\s:\radius);
\end{scope}
	
}

\newcommand{\drawnumbers}{
  \foreach \i in {1,2,...,\n} {
  \pgfmathsetmacro{\x}{\angle*\i}
  \draw (\x:\radius*1.15) node {\footnotesize \i};
}
}

\newcommand{\boundA}[3]{
	\pgfmathsetmacro{\r}{#1}
	\pgfmathsetmacro{\bumpr}{#2}
	\pgfmathsetmacro{\destination}{#3}
	\pgfmathsetmacro{\perturbe}{\angle/\n}
	\path [name path=polyedge1] (\angle*\r:\radius) -- (\angle*\r + \angle:\radius);
	\path [name path=radius1] (0:0) -- (\angle*\r + \angle/2 + \bumpr*\perturbe:\radius);
	\draw[->,
	name intersections={of=polyedge1 and radius1,by=p},
	shorten >=\radius*0.1 cm] (p) ++(\angle*\r + \angle/2 + \bumpr*\perturbe:\radius*0.15) -- (\angle*\destination: \radius*1.15);

}

\newcommand{\boundB}[3]{
	\pgfmathsetmacro{\rangle}{#1*\angle + \angle/2 + #2*\angle/\n}
	\pgfmathsetmacro{\sangle}{#1*\angle + \angle/2 + #3*\angle/\n}

	\draw[->,shorten <=\radius*0.02cm,shorten >=\radius*0.05cm] (\rangle:\radius*1.05) -- (\sangle:\radius*1.05);

}

\newcommand{\makediag}[8]{
	\begin{tikzpicture}[rotate=60,baseline=(current bounding box.east)]
	\begin{scope}
	\drawWLD{6}{0.8}
	\drawprop{#1}{#2}{#3}{#4}
	\drawprop{#5}{#6}{#7}{#8}
	\end{scope}
	\end{tikzpicture}
}

\newcommand{\R}{\mathbb{R}}

\newcommand{\Gr}{\mathbb{G}_{\R, \geq 0}}

\newcommand{\D}{\partial}

\newcommand{\rd}{\textrm{d}}
\newcommand{\Res}{\textrm{Res}}

\def\ba #1\ea{\begin{align} #1 \end{align}}
\def\bas #1\eas{\begin{align*} #1 \end{align*}}
\def\bml #1\eml{\begin{multline} #1 \end{multline}}
\def\bmls #1\emls{\begin{multline*} #1 \end{multline*}}

\newcommand{\cP}{\mathcal{P}}

\newcommand{\cI}{\mathcal{I}}
\newcommand{\cC}{\mathcal{C}}
\newcommand{\cB}{\mathcal{B}}
\newcommand{\G}{\mathbb{G}}
\newcommand{\Prop}{\textrm{Prop}}
\newcommand{\cW}{\mathcal{W}}

\newcommand{\cZ}{\mathcal{Z}}

\newcommand{\BB}{\mathcal{B}}
\newcommand{\CS}{\mathcal{S}}

\newtheorem{thm}{Theorem}[section]
\newtheorem{conj}[thm]{Conjecture}

\newtheorem{cor}[thm]{Corollary}
\newtheorem{prop}[thm]{Proposition}
\newtheorem{mnemonic}[thm]{Graphical Prompt}

\theoremstyle{remark}
\newtheorem{eg}[thm]{Example}

\theoremstyle{definition}
\newtheorem{dfn}[thm]{Definition}
\newtheorem{rmk}[thm]{Remark}
\newtheorem{ntn}[thm]{Notation}

\begin{document}

\title{A study in $\Gr(2,6)$: from the geometric case book of Wilson loop diagrams and SYM $N=4$}
\author{S. Agarwala and S. Fryer}
\date{}
\maketitle

\begin{abstract}
We study the geometry underlying the Wilson loop diagram approach to calculating scattering amplitudes in the gauge theory of Supersymmetric Yang Mills (SYM) $N=4$. By applying the tools developed to study total positivity in the real Grassmannian, we are able to systematically compute with all Wilson loop diagrams of a given size and find unexpected patterns and relationships between them.
We focus on the smallest nontrivial multi-propagator case, consisting of 2 propagators on 6 vertices, and compute the positroid cells associated to each diagram and the homology of the subcomplex they generate in $\Gr(2,6)$. We also verify in this case the conjecture that the spurious singularities of the volume functional {\em do} all cancel on the codimension 1 boundaries of these cells.
\end{abstract}

\section{Introduction}\label{sec:intro}

In this paper, we consider the geometric underpinnings of the scattering amplitudes of SYM ${N=4}$ theory. A close association between the SYM $N=4$ theory and polytopes in the positive Grassmannians was established with the advent of the Amplituhedron \cite{Arkani-Hamed:2013jha}. There, the on shell scattering amplitudes in SYM $N=4$ are computed using BCFW diagrams. Another approach to understanding the amplitudes of SYM $N=4$ theory is via the study of MHV amplitudes. This is the approach taken in this paper, where the MHV amplitudes are represented by Wilson loop diagrams. In both these cases the scattering amplitudes are given in terms of volumes polytopes. The actual volume is computed as a function of the external particle data, here represented by a matrix $\cZ$.

Conjecturally, the Amplituhedron is tiled by the image (under multiplication by $\cZ$) of the $4k$-dimensional cells of $\Gr(k,n)$ that correspond to the BCFW diagrams in the physical theory \cite{Arkani-Hamed:2013jha,AmplituhedronDecomposition}. There is a large body of work connecting BCFW diagrams to plabic graphs \cite[Chapter 4]{cyclichyperplanes,AmplituhedronDecomposition,GrassmannAmplitudebook}, and from there to a stratification by positroid cells of a subspace of $\Gr(k, n)$ \cite{Postnikov}.

In this paper, we follow the Wilson loop diagram approach to understanding SYM $N=4$ theory. While BCFW diagrams give the on shell scattering amplitudes of SYM $N=4$ theory, the Wilson loop gives the entire scattering amplitude \cite{Alday:2007hr}. The calculations done in this paper consider $N^kMHV$ diagrams in twistor space \cite{Boels:2007qn, Bullimore:2009cb, Mason:2010yk} which we call Wilson loop diagrams. These are the equivalent of Feynman diagrams for the theory in the twistor space setting. One of the ultimate goals of this program is to understand the geometry underlying the Wilson loop diagrams, and the space they are mapped to by the external data matrices $\cZ$, just as in the Amplituhedron. There is some work already in relating the geometry defined by the Wilson loop diagrams to the Amplituhedron \cite{Amplituhedronsquared}, but there are several barriers that make understanding this geometric relationship difficult.

For instance, the biggest difference between the Wilson loop diagram and the BCFW setting is the need for a gauge vector. As noted above, the BCFW diagrams correspond to $4k$-dimensional subspaces of the positive Grassmannian $\Gr(k, n)$. The $N^kMHV$ diagrams depend on a choice of gauge, and if we include the gauge structure then the Wilson loop diagrams define $4k$-dimensional subspaces of the full Grassmannian $\G_{\R}(k, n+1)$. However, if we omit the gauge data, the Wilson loop diagrams tile a (conjecturally $3k$-dimensional) subspace of the positive Grassmannian $\Gr(k,n)$ \cite{Wilsonloop}.

On the one hand, this $3k$-dimensional subspace seems to contain much of the physically interesting structure of the Wilson loop diagrams, in that all the cancellations of singularities in the scattering amplitude computation happen on the boundaries of the $3k$-dimensional CW complex induced on this subspace by the positroid decomposition of $\Gr(k,n)$. In addition, tools from total positivity are extremely effective for analyzing the behavior of this complex and identifying patterns not previously observed. However, one cannot interpret the scattering amplitude as the volume of the image as this space under the external data; in order to give a geometric interpretation of the amplitude, one must consider the gauge data as well. Thus the $3k$-dimensional approach is but one part of the Wilson loop diagram story; for another approach, see \cite{Amplituhedronsquared}.

Before proceeding, it is worth noting that while twistor space is complex, throughout this paper (and the literature) one speaks of {\em positivity} in Grasmanninians. This is accomplished via a simplification first introduced by Arkani-Hamed \cite{Arkani-Hamed:2013jha}. Instead of representing the $n$ external fermionic particles as sections of a complex line bundle over a complex twistor space, we work with a projection, or bosonification, onto a $k$-dimensional real vector bundle over a $4$-dimensional real space-time. That is, we associate to each external particle a vector $Z_i \in \R^{4+k}$, which represents a bosonized fermion. We decompose these vectors as $Z_i = (Z_i^\mu, \vec{z}_i)$ with $Z_i^\mu \in \R^4$ encoding the momentum of the particles, and $\vec{z}_i \in \R^k$. Furthermore, the set of vectors $\{Z_1 \ldots Z_n\}$ are chosen such that any ordered subset of $4+k$ vectors defines a positive volume. These vectors comprise the rows of the matrix $\cZ$.

In this paper our goals are twofold. The first is to survey some of the general theory of Wilson loop diagrams, Feynman integrals, and the applications of tools from total positivity (e.g. Le diagrams, vertex-disjoint path systems) to these questions, in a way that is accessible to both mathematicians and physicists. The second goal is to explicitly compute the details of a small and concrete example, namely the case $k=2$, $n=6$. This is the smallest multi-propagator interaction in this theory, yet even this case had not been described in detail until now.

By applying the techniques of total positivity, we verify that all admissible $N^2MHV$ diagrams on 6 points yield positroids cells of dimension 6 (= $3k$); a complete list of the these positroids is given in Table \ref{nameWLDLetable}. With this data in hand, we are able to obtain a complete description of the physically-interesting boundaries of the Wilson loop diagrams (the ``boundary diagrams'') in terms of the positroid cell structure on $\Gr(2,6)$ (see Section \ref{mnemonic} for precise definitions):

\begin{prop}(Proposition \ref{res:bounds})
$B$ is a 5-dimensional positroid cell parametrized by the boundary diagram $\D_{p,v}(W)$ of some admissible Wilson loop diagram $W$ and corresponding to a simple pole of the associated integral $\cI(W)(\cZ_*)$, if and only if $B$ lies on the boundary of 6-dimensional positroid cells in $\Gr(2,6)$ that are parametrized by at least two distinct Wilson loop diagrams.
\end{prop}

Not every 5-dimensional positroid in $\Gr(2,6)$ corresponds to a cell parametrized by a boundary diagram. This can be seen by direct computatation (in this ``small'' case, there are only 50 cells of dimension 5 to consider), or by computing the homology of the space $\cW(2,6)$ generated by the cells corresponding to the Wilson loop diagrams:

\begin{thm}(Theorem \ref{res:homology})
The homology groups of $\cW(2,6)$ are as follows: \bas H_i(\cW(2,6)) =\begin{cases} \R & \textrm{if } i \in 0, 5; \\ 0 & \textrm{else}.  \end{cases} \eas
\end{thm}

As noted above, it is conjectured that the singularities appearing in the computation of the scattering amplitude should all cancel out on the codimension 1 boundaries of the cells associated to Wilson loop diagrams. In Section \ref{sec:cancellingpoles} we verify this for $k=2$, $n=6$, as summarized in the following theorem.

\begin{thm} (Theorem \ref{boundarycancelthm})
Let $W$ be an admissible Wilson loop diagram with $2$ propagators and $6$ external particles. Let $\Sigma'$ be any $5$-dimensional boundary of the $6$-dimensional cell $\Sigma(W)$ associated to $W$. Then
\ba \sum_{\begin{subarray}{c} \D_{p,v}(W') = \Sigma', \\ W' \textrm{ admiss. } \end{subarray}} \Res_{\Delta_{p,v}(W') \rightarrow 0}\cI(W')(\cZ_*) =0\;.\ea\end{thm}

Even in this small case, we observe several curiosities and obstacles that will carry forward as one adds more propagators and interacting particles. In particular, we find examples of multiple Wilson loop diagrams corresponding to the same positroid cell, as shown in \cite{Wilsonloop}. This leads to positroid cells  sharing boundaries that are apparent from the CW complex and the Le diagrams, but not from the Wilson loop diagrams themselves. The multiplicity of the map from Wilson loop diagrams to the positroid cells leads to many cases of three or more distinct Wilson loop diagrams sharing the same codimension 1 boundary cell with each other. We also observe instances of two different diagrams sharing two distinct boundaries with each other. In short, the CW structure of $\cW(2,6)$ is far more complicated than can be easily inferred from the Wilson loop diagrams themselves, and can only be truly studied with the help of tools from total positivity such as Le diagrams.

Many of the results in this paper are obtained by direct calculation, using objects such as Le diagrams primarily to organize results and observe patterns. This barely scratches the surface of what the combinatorial tools can achieve; in a forthcoming paper, we explore the applications of combinatorics to Wilson loop diagrams for general $k$ and $n$.

This paper is organized as follows. In Section \ref{sec:WLDbackground}, we define Wilson loop diagrams from a purely combinatorial point of view. We outline the tools and results that we use from the theory of total positivity, including positroid cells, Le diagrams, and vertex-disjoint path systems. Finally, we show how to identify the positroid cell associated to a given Wilson loop diagram.

In Section \ref{sec:Feynmanrules} we recall the integrals associated to Wilson loop diagrams and give a brief discussion of how these integrals correspond to the holomorphic Wilson loop. We also include a worked example of evaluating the integral associated to a Wilson loop to illustrate the computations involved, and to motivate later sections, In Section \ref{sec:2,6case} we consider the specifics of $N^2MHV$ diagrams on $6$ points, and give a full description of the positroid cells they define. We use this to compute the homology of the CW complex generated by the Wilson loop diagrams, and to give a geometric interpretation of their physically-interesting boundaries (i.e. the so-called ``graphical boundaries'').

Finally, in Section \ref{sec:cancellingpoles} we show that all singularities of the integrals associated to $N^2MHV$ Wilson loop diagrams on 6 points cancel in the sum over all diagrams. Furthermore, the cancellations for these diagrams all occur, as expected, on the codimension $1$ boundaries of the cells. This behavior has been conjectured for general $k$ and $n$, but had not previously been verified even for $k=2$, $n=6$ due to the difficulties in identifying all shared codimension 1 boundaries from the Wilson loop diagrams alone.

{\bf Acknowledgements}: Work on this paper was begun during the conference ``Total Positivity: A bridge between Representation Theory and Physics'' hosted by the University of Kent, and both authors are grateful for the opportunities and connections sparked by it.

Agarwala is indebted to Paul Heslop and Alastair Stewart for many critical discussions about integrals associated to Wilson loop diagrams. Some of the calculational maneuvers used to compare integrals along boundaries presented in this paper come directly from examples worked out in correspondences and conversations with him.

Fryer gratefully acknowledges the hospitality of the University of Kent mathematics department, who provided a welcoming and supportive environment on several occasions during the writing of this paper, and thanks Matthew Towers for providing an easy-to-code algorithm for testing containment between positroid cells.

\section{Wilson loop diagrams \label{WLDs}}\label{sec:WLDbackground}

This section is an introduction to the combinatorics of Wilson loop diagrams for mathematicians. As such, we omit the derivation of the diagrams and associated integrals, and the precise definitions of the physical objects involved (see existing literature \cite{Adamo:2011cb}), and focus on the diagrams themselves as combinatorial objects.

In Section \ref{sub:notation}, we establish the notation and conventions used in this paper. Section \ref{sub:WLDdef} describes the diagramatics of the Wilson loop diagrams themselves, Section \ref{sub:positroidbackground} provides an overview of total positivity and some related combinatorial objects, and Sections \ref{sub:admissibleWLD} and \ref{sub:WLDpositroids} make precise the relationship between Wilson loop diagrams and total positivity.

\subsection{Notation}\label{sub:notation}

Let $M_{\R}(k,n)$ denote the set of all $k\times n$ real matrices, and $M_{\R,+}(k,n)$ the subset of those matrices whose maximal (i.e. $k\times k$) minors are strictly positive.

We represent the Grassmannian $\G_\R(k,n)$ as the set of all full rank $k\times n$ real matrices modulo the left action of $GL_k(\R)$; intuitively, these are $k\times n$ matrices with linearly independent rows, where two matrices are ``the same'' if we can get from one to the other by performing row operations.

\begin{dfn}A point in $\G_\R(k,n)$ is called {\bf totally nonnegative} if it can be represented by a matrix whose maximal minors (i.e. its $k\times k$ minors) are all nonnegative. Denote the set of all totally nonnegative points by $\Gr(k,n)$. \end{dfn}

Let $[n]$ denote the interval $\{1, 2, \dots, n\}$, and
\[\binom{[n]}{k} = \big\{ I \subset [n] \ \big |\ |I| = k\big\}\]
the set of all $k$-subsets of $[n]$. Clearly $\binom{[n]}{k}$ is in bijection with the set of possible maximal minors of a $k\times n$ matrix. We write $\Delta_I$ for the minor with columns indexed by $I$, or $\Delta_I(A)$ for the value of the minor evaluated on a specific matrix $A$.

Each individual Wilson loop diagram encodes one piece of the calculation whose sum yields the scattering amplitude. The result of this calculation depends on the external data associated to each particle in the interaction.

\begin{dfn}
An external particle of a Wilson loop diagram is denoted $Z \in \R^{4+k}$. Write $Z^j$ to denote the $j^{th}$ component of $Z$ and $Z^{\mu}$ to denote the projection of the section $Z$ onto its twistor component, i.e. onto the first four coordinates.
\end{dfn}

An $n$ point Wilson loop diagram has $n$ external particles and a gauge vector, and we collect this external data in a matrix as follows.

\begin{dfn}\label{def:cZ def} Write $\cZ$ for the $n \times (k+4)$ matrix whose rows consist of the $n$ vectors $Z_1,Z_2, \dots, Z_n$ associated to the external particles. We restrict our attention to collections of particles in general position, i.e. such that $\cZ \in M_{\R,+}(n,4+k)$.  We may choose a {\em gauge vector} $Z_*$ that is a point in a zero section of the line bundle of external particles. Denote by $\cZ_* \in M_{\R}(n+1,4+k)$ the matrix obtained by prepending $Z_*$ above the first row of $\cZ$. \end{dfn}

We refer to these matrices $\cZ_*$ as the matrices of external data. A standard convention throughout this paper is that objects with a star subscript include the gauge vector data, while those without a star subscript omit the gauge vector from consideration. As described in the introduction, there are often reasons for including or omitting the gauge from our calculations.

We also define a notation for certain $4\times 4$ determinants which play an important role in calculations:
\[\langle a b c d \rangle = \det(Z_a^\mu Z_b^\mu Z_c^\mu Z_d^\mu),\]
i.e. the determinant of the $4\times 4$ matrix whose rows correspond to the external particles $Z_a$, $Z_b$, $Z_c$, and $Z_d$, each projected to their first four coordinates.

\subsection{Wilson loop diagrams}\label{sub:WLDdef}

We now discuss the Wilson loop diagrams combinatorially.

\begin{dfn}\label{WLDdef}
A {\bf Wilson loop diagram} is a convex $n$-gon ($n\geq 4$) with $k$ internal (wavy) lines called MHV propagators, each of which connects a pair of distinct polygon edges. See Figure \ref{fig:WLDexample} for an example of a Wilson loop diagram.
\end{dfn}

\begin{figure}[h!]
\bas W =
{\begin{tikzpicture}[rotate=67.5,baseline=(current bounding box.east)] 
\begin{scope}
\drawWLD{8}{1.5} 
\drawnumbers 
\drawprop{2}{0}{4}{0} 
\drawprop{4}{1}{7}{1} 
\drawprop{5}{0}{7}{-1}
\end{scope}
\end{tikzpicture} } \eas
\caption{A Wilson loop diagram with 8 vertices and 3 propagators.}
\label{fig:WLDexample}
\end{figure}
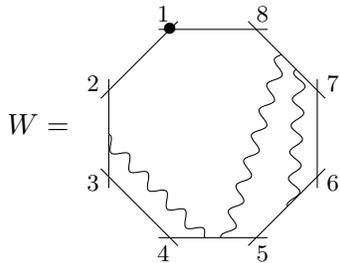

The vertices of the boundary polygon correspond to the external particles involved in the represented interaction. The location of the propagator endpoint on the polygon edge is not significant; if several propagators share an edge, by convention we arrange them in order to minimize the number of crossings.

Note that we have not yet put any restrictions on the behaviour of the propagators, e.g. excluding diagrams with crossing propagators. The underlying physics will impose certain conditions on the propagators; this is described in Section \ref{sub:admissibleWLD} below.

By labeling the vertices of the polygon, we obtain the following equivalent definition of a Wilson loop diagram:
\begin{dfn}\label{WLDdef2}
A Wilson loop diagram (WLD) is comprised of a cyclicly ordered set $(S,<) \subset [n]$ and a set of $k$ pairs of elements of $S$. For ease of reference, we impose the following convention on the propagators: \bas \cP = \{(i_1, j_1), \ldots (i_k, j_k)\ |\ i_r, j_r \in S; i_r <_1 j_r\ \forall r\}\;, \eas
where $<_1$ denotes the total order induced from $<$ by choosing some element of $S$ (usually 1) to be the ``first'' element. We write $W = (\cP,S)$ for the Wilson loop diagram, or (for simplicity) $W = (\cP,n)$ when $S = [n]$.
\end{dfn}

The set $S$ labels the boundary vertices of a convex $n$-gon, starting at the marked vertex and moving counterclockwise. This induces a labelling of the edges, where the $i$th edge connects vertices $i$ and $i+1$ (or the successor of $i$, if $S \neq [n]$). The ordered pair $(i_r,j_r) \in \cP$ corresponds to the propagator connecting edges $i_r$ and $j_r$.

For example, the diagram in Figure \ref{fig:WLDexample} can also be written as $W = (\{(2,4), (4,7), (5,7)\}, 8)$.

\begin{dfn}\label{Propsupportdef}
\begin{enumerate}
\item Given a propagator $p \in \cP$, the {\bf support} of $p$ is the set ${V_p = \{i_p, i_p+1, j_p, j_p+1\}}$, i.e. the endpoints of the two boundary edges that the propagator touches.
\item Given a set of propagators $Q \subset \cP$, define $V_Q = \cup_{q \in Q} V_q$ to be the support of $Q$.
\item For a set of vertices $V \subset [n]$, write $\Prop(V) = \{p \in \cP\ |\  V \cap V_p \neq \emptyset\}$ for the subset of propagators supported by $V$.
\end{enumerate}
\end{dfn}

\begin{dfn}
If $W = (\cP,n)$ is a Wilson loop diagram, and $Q \subset \cP$ a non-empty subset of its propagators, the {\bf subdiagram} of $W$ associated to $Q$ is $W_Q = (Q,V_Q)$.
\end{dfn}

For ease of reference later, we impose an (arbitrary) naming convention for the propagators. Namely, given a Wilson loop diagram $W$ with $k$ propagators, label the propagators as $p_1,\dots, p_k$, in the order they are first encountered while walking counterclockwise around $W$ from the marked vertex. We use this convention unless indicated otherwise.

We now associate two matrices to each Wilson loop diagram, based on the ordering convention above. All work in this paper can be done under any other ordering convention.

\begin{dfn} \label{Cstardef}
Define a $k \times n$ matrix $\cC(W)$ by placing indeterminates $c_{b,a}$ in some of its entries and 0s elsewhere, as follows:
\bas
\cC(W)_{b,a} = \begin{cases} c_{b, a}  & \textrm{if } a\in V_{p_b}; \\ 0 & \textrm {else .} \end{cases}
\eas
We will also want to consider the matrix $\cC_*(W)$ obtained by prepending a column of 1s to $\cC(W)$; in order to maintain column labelling consistency, we label the columns of $\cC_*(W)$ from 0 to $n$. In other words,
\bas
\cC_*(W)_{b, a} = \begin{cases}1 &  \textrm{if } a = 0;  \\ c_{b,a}  & \textrm{if } a\in V_{p_b}; \\ 0 & \textrm {else .} \end{cases}
\eas

In both matrices, the entries $c_{b,a}$ are real indeterminates.
\end{dfn}

\begin{eg}
For the Wilson loop diagram $W = (\{(2,4), (4,7),(5,7)\},8)$ in Figure \ref{fig:WLDexample}, the associated matrices $\cC(W)$ and $\cC_*(W)$ are
\begin{gather*}
\cC(W)  = \begin{pmatrix}
0 & c_{1,2} & c_{1,3} & c_{1,4} & c_{1,5} & 0 & 0 & 0 \\
0 & 0 & 0 & c_{2,4} & c_{2,5} & 0 & c_{2,7} & c_{2,8} \\
0 & 0 & 0 & 0 & c_{3,5} & c_{3,6} & c_{3,7} & c_{3,8}
\end{pmatrix}, \\
\cC_*(W)  = \begin{pmatrix}
1 & 0 & c_{1,2} & c_{1,3} & c_{1,4} & c_{1,5} & 0 & 0 & 0 \\
1 & 0 & 0 & 0 & c_{2,4} & c_{2,5} & 0 & c_{2,7} & c_{2,8} \\
1 & 0 & 0 & 0 & 0 & c_{3,5} & c_{3,6} & c_{3,7} & c_{3,8}
\end{pmatrix}.
\end{gather*}
\end{eg}

\subsection{Positroids and Le diagrams}\label{sub:positroidbackground}

In recent years, the study of totally nonnegative matrices and their associated combinatorics has emerged as an extremely useful tool for the study of scattering amplitudes in SYM $N=4$; see for example \cite{Wilsonloop,Arkani-Hamed:2013jha,cyclichyperplanes}. In this section we give a short introduction to this nonnegative viewpoint, and outline the key definitions and techniques we make use of in this paper.

For the purposes of this paper, a {\bf representable matroid} is any collection $\cB \subseteq \binom{[n]}{k}$ which can be ``represented'' by an element of $\G_{\R}(k,n)$, i.e. there exists some $A \in \G_{\R}(k,n)$ such that for all $B \in \binom{[n]}{k}$,
\[\Delta_B(A) \neq 0 \iff B \in \cB.\]
Following the language of matroid theory, we will refer to the elements of $\cB$ as the {\bf bases} of the matroid.

\begin{dfn}\label{def:positroid}
A {\bf positroid} is a representable matroid $\cB$ which can be represented by an element of $\Gr(k,n)$.
\end{dfn}
This definition induces a stratification of $\Gr(k,n)$ into {\bf positroid cells}: for each positroid $\cB$, define
\[\CS_{\BB} = \left\{ A \in \Gr(k,n) \ | \ \Delta_B (A) \neq 0 \text{ if and only if } B \in \BB\right\},\]
i.e. the set of points in $\Gr(k,n)$ which represent $\cB$.

These definitions were introduced the foundational 2006 preprint \cite{Postnikov} by Postnikov. The positroid stratification of $\Gr(k,n)$ has the structure of a regular CW complex \cite[Theorem 1.1]{GrBall}; in other words, $\Gr(k,n)$ is homeomorphic to a closed ball, the positroids partition $\Gr(k,n)$ into convex open cells of dimension $0 \leq d \leq k(n-k)$, the closure of a cell of dimension $d$ is the union of that cell with finitely many cells of dimension $\leq d-1$, and explicit attaching maps have been constructed, e.g. \cite[Theorem 6.2]{PSW}.

\begin{rmk}\label{boundaries via bases}
If we are only interested in which postroid cells lie on the boundary of a given cell, and not in the precise attaching map data, we may consider the face poset of this complex instead. Specifically: the positroid cell $\CS_{\BB'}$ is contained in the boundary of the cell $\CS_{\BB}$ if and only if the there is an inclusion $\BB' \subset \BB$ of the bases.
\end{rmk}

Positroids have many nice combinatorial properties, see for example \cite{PositroidsNoncrossing,GrBall,Postnikov}. While the matrices in this paper are small enough that it is easy to compute their minors by hand (and we will often use this method when considering the cell complex structure), examining long lists of bases is often unilluminating. We therefore introduce  Le diagrams (Definition \ref{def:Lediagram}) as a convenient method of labelling the positroid cells, and describe the method of vertex-disjoint path systems (Theorem \ref{res:bases from vdps}) for passing between the list of bases defining a positroid and its Le diagram.

\begin{dfn}\label{def:Lediagram}
A \textbf{Le diagram} is a Young diagram where each box contains either a $+$ or a $0$ symbol, subject to the rule that if a box contains a $0$ then at least one of the following holds:
\begin{itemize}
\item Every box to its left (in the same row) also contains a $0$; {\em or}
\item Every box above it (in the same column) also contains a $0$.
\end{itemize}
\end{dfn}
For example,
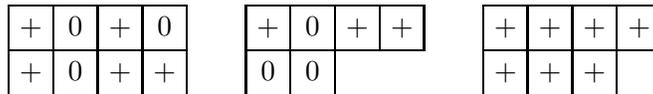
\[\ytableaushort{+0+0,+0++} \qquad  \ytableaushort{+0++,00} \qquad \ytableaushort{++++,+++}\]
are all Le diagrams, while
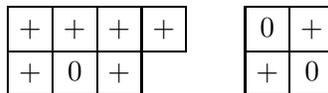
\[\ytableaushort{++++,+0+} \qquad \ytableaushort{0+,+0}\]
are not.

The positroids in $\Gr(k,n)$ are in bijection with the Le diagrams that fit inside a $k \times (n-k)$ rectangle, i.e. have at most $k$ rows and at most $n-k$ columns \cite[Theorem 6.5]{Postnikov}.

\begin{rmk}
The dimension of a positroid cell is simply the dimension of the open ball it is homeomorphic to. Given a positroid cell, its dimension can easily be read from the corresponding Le diagram: it is precisely the number of $+$ squares appearing in the Le diagram.
\end{rmk}

There is also an algorithm to reconstruct the set $\BB$ of nonzero minors from the Le diagram, as we now describe. Given a Le diagram $L$, construct its associated graph $\Gamma(L)$ as follows:
\begin{enumerate}
\item Label each step along the southeast border of $L$ with the numbers $1, 2, \dots, n$. (Note that if $L$ has fewer than $k$ rows, or fewer than $(n-k)$ columns, some of these steps will lie on the bounding rectangle; see Figure \ref{fig:gammaL}.)
\item Place a vertex in each $+$ square of $L$, and an additional vertex next to each row and column label.
\item Join any two consecutive vertices in the same row with an arrow directed leftwards, and any two consecutive vertices in the same column with an arrow directed downwards.
\end{enumerate}
See Figure \ref{fig:gammaL} for an example. It follows directly from the Le property that $\Gamma(L)$ is planar, i.e. two arrows can only meet at a vertex.
\begin{figure}[h!]
\begin{center}
\begin{tabular}{cc}

$L = \ytableaushort{+0+0,+0++,0+}$, \qquad \qquad

&

$\Gamma(L) =$ \begin{tikzpicture}[baseline=(current bounding box.east)]
\draw[dotted, gray] (2,0) -- ++(3,0) -- ++(0,3) -- ++(-1,0);
\draw[gray] (0,0) grid +(2,1);
\draw[gray] (0,1) grid +(4,1);
\draw[gray] (0,2) grid +(4,1);

\node (A) at (0.5,2.5) {$\bullet$};
\node (B) at (2.5,2.5) {$\bullet$};
\node (C) at (0.5,1.5) {$\bullet$};
\node (D) at (2.5,1.5) {$\bullet$};
\node (E) at (3.5,1.5) {$\bullet$};
\node (F) at (1.5,0.5) {$\bullet$};

\node (0) at (4.6,2.75) {$1$};
\node (1) at (4.25,2.5) {$2$};
\node (2) at (4.25,1.5) {$3$};
\node (3) at (3.5,0.75) {$4$};
\node (4) at (2.5,0.75) {$5$};
\node (5) at (2.25,0.5) {$6$};
\node (6) at (1.5,-0.5) {$7$};
\node (7) at (0.5,-0.5) {$8$};

\draw [->] (B) to (A);
\draw [->] (E) to (D);
\draw [->] (D) to (C);
\draw [->] (A) to (C);
\draw [->] (B) to (D);

\draw [->] (1) to (B);
\draw [->] (2) to (E);
\draw [->] (E) to (3);
\draw [->] (D) to (4);
\draw [->] (5) to (F);
\draw [->] (F) to (6);
\draw [->] (C) to (7);

\end{tikzpicture}
\end{tabular}
\end{center}
\caption{Constructing $\Gamma(L)$ from $L$ in $\Gr(3,8)$.}
\label{fig:gammaL}
\end{figure}

Write $S$ for the set of source vertices in the diagram (i.e. the vertices attached to row labels), and $T$ for the set of target vertices (column labels). A \textbf{path} in $\Gamma(L)$ is any path from a vertex $s \in S$ to a vertex $t \in T$ along these directed arrows.  Two paths are called \textbf{vertex-disjoint} if they do not have any vertices in common.

Let $I = \{i_1, \dots, i_r\}$ be a subset of the row labels, and $J = \{j_1, \dots, j_r\}$ a subset of the column vertices. A \textbf{vertex-disjoint path system} for $(I,J)$ is a collection of paths
\[i_1 \rightarrow j_1,\ i_2 \rightarrow j_2,\ \ldots, \ i_r \rightarrow j_r,\]
which are pairwise vertex-disjoint.

\begin{thm}\label{res:bases from vdps}
Let $L$ be a Le diagram specifying a cell in $\Gr(k,n)$, and construct its graph $\Gamma(L)$ as above.  Then $B$ is a basis for the positroid corresponding to $L$ if and only if $|B| = k$ and $(S \backslash B, T \cap B)$ admits a vertex-disjoint path system.
\end{thm}
\begin{proof}
Combine \cite[Theorem 5.6]{MR2774624} and \cite[Theorem 4.2]{GLL2}.
\end{proof}

\begin{eg}
For the Le diagram in Figure \ref{fig:gammaL}, we have $S = \{2,3,6\}$ and $T = \{1,4,5,7,8\}$. Then for example, $B = \{4,5,6\}$ is a basis for this positroid, since
\[S \backslash B = \{2,3\}, \quad T \cap B = \{4,5\},\]
and there is a vertex-disjoint path system $2\rightarrow 5$, $3 \rightarrow 4$ in $\Gamma(L)$. On the other hand, $\{3,4,5\}$ is {\em not} a basis, since there is no vertex-disjoint path system from $\{2,6\}$ to $\{4,5\}$ in $\Gamma(L)$.
\end{eg}

\begin{ntn}
We will often drop the set notation for bases where this will not cause any confusion, i.e. write $456$ for $\{4,5,6\}$.
\end{ntn}

For small examples in particular, vertex-disjoint path systems provide an effective way of listing all of the bases of a positroid without writing down an explicit matrix belonging to that positroid. Conversely, given a list of bases for a positroid, we can use Theorem \ref{res:bases from vdps} to reconstruct the corresponding Le diagram; this allows us to start with a Wilson Loop diagram and construct its Le diagram via the matrix $\cC(W)$.

Recall that the cell structure on the positroids of $\Gr(k,n)$ can be seen in its face poset (ordered by inclusion). It is important to emphasize that these inclusions {\em cannot} always be read directly from the Le diagrams, as the following example demonstrates.

\begin{eg}\label{eg:badLebehavior} In $\Gr(2,4)$, consider
\[ A_1 = \begin{pmatrix}1 & 0 & 0 & -a \\ 0 & 1 & 0 & b \end{pmatrix}, \quad A_2 = \begin{pmatrix}1 & 0 & -c & -c \\ 0 & 1 & d & d\end{pmatrix}, \quad a, b,c, d >  0\]
$A_1$ belongs to the positroid $\CS_{\BB_1}$ with $\BB_1 = \{12,14,24\}$, while $A_2$ belongs to $\CS_{\BB_2}$ with $\BB_2 = \{12,13,14,23,24\}$. So we certainly have $\BB_1 \subseteq \BB_2$, but there is no obvious relation between the Le diagrams:
\[L(\CS_{\BB_1}) = \ytableaushort{+0,+0}, \qquad L(\CS_{\BB_2}) = \ytableaushort{0+,++}.\]
\end{eg}

Thus we use Le diagrams as a convenient method of labelling positroids and calculating their dimension, and lists of bases when examining the cell structure.

\subsection{Admissible Wilson loop diagrams}\label{sub:admissibleWLD}

Thus far, we have not put any conditions on the behavior of the Wilson loop diagrams. It turns out that restricting our attention to the physically interesting ones (the {\em admissible} diagrams, Definition \ref{admisWLD}) yields matrices which also have interesting positivity conditions.

Fix a matrix of external data $\cZ_*$, as described in Section \ref{sub:notation}. By Cramer's rule, each matrix $\cC(W)$ defines a unique kernel of the matrix $\cZ^\mu_*$. In \cite{Wilsonloop}, the first author and Amat show that solving the equation \begin{equation}\label{eq:CZ = 0}\cC_*(W) \cdot \cZ^\mu_* = \mathbf{0}\end{equation} for the entries $c_{b,a}$ of $\cC_*(W)$ yields expressions that can be written in terms of certain minors of $\cZ^\mu_*$, as we now describe. (The significance of equation \eqref{eq:CZ = 0} is discussed following Theorem \ref{admisWLDSarepositroids}.)

\begin{dfn}\label{sigmadfn}
For a propagator $p_b = (i_b, j_b)$, and supporting vertex $a \in V_{p_b}$, define $\sigma_{b,a}$ to be the determinant formed by replacing the vector $Z_a^\mu$ in the determinant $\langle i_b\ (i_b+1)\ j_b\ (j_b+1)\rangle$ with the gauge vector $Z_*^\mu$, i.e.
\bas  \sigma_{b,a} = \langle  Z^\mu_{i_b} \ldots Z^\mu_* \widehat{Z^\mu_a} \ldots Z^\mu_{j_b+1} \rangle \;. \eas  \end{dfn}
Let $\cC_*(W)(\cZ_*)$ be the matrix obtained from $\cC_*(W)$ by replacing each $c_{b,a}$ with \begin{equation}\label{Cmatrixentries}\frac{\sigma_{b,a}}{\langle i_b\ (i_b+1)\ j_b\ (j_b+1)\rangle},\quad 1 \leq b \leq k, \quad a \in V_{p_b}. \end{equation}
It now follows from Cramer's rule that $C_*(W)(\cZ_*) \cdot \cZ_*^\mu = \mathbf{0}$ \cite[Equation (4)]{Wilsonloop}. Thus, $\cC_*(W)(\cZ_*)$ is exactly the matrix that solves equation \eqref{eq:CZ = 0}, for fixed external data $\cZ_*$. See Example \ref{complicateddiagrameg} for a worked example of computing $\cC_*(W)(\cZ_*)$.

We are interested in Wilson loop diagrams that define positroid cells of $\Gr(k,n)$. To this end, we give the following definition.

\begin{dfn} \label{admisWLD}
A Wilson loop diagram $W = (\cP, n)$ is called {\bf admissible} if it satisfies the following conditions:
\begin{enumerate}
\item $n \geq |\cP| +4$;
\item There does not exist a (non-empty) subset of propagators $Q \subseteq \cP$ such that $|V_Q| < |Q| + 3$;
\item $W$ has no crossing propagators.
\end{enumerate}
\end{dfn}
It is more illuminating to see examples of Wilson loop diagrams which {\em fail} to be admissible; some examples are listed in Figure \ref{fig:nonadmissibleWLD}.

\begin{figure}[h]
\begin{center}
\begin{tabular}{ccccccc}
\begin{tikzpicture}[rotate=60,baseline=(current bounding box.east)]
	\begin{scope}
	\drawWLD{6}{1.5}
	\drawprop{1}{-1}{6}{-2}
	\drawprop{1}{-0.5}{5}{0}
	\drawprop{1}{0}{4}{0}
	\drawprop{1}{0.5}{3}{0}
		\end{scope}
	\end{tikzpicture}

& \quad &

\begin{tikzpicture}[rotate=60,baseline=(current bounding box.east)]
	\begin{scope}
	\drawWLD{6}{1.5}
	\drawprop{1}{-1}{4}{1}
	\drawprop{1}{1}{4}{-1}
		\end{scope}
	\end{tikzpicture}

& \quad &

\begin{tikzpicture}[rotate=60,baseline=(current bounding box.east)]
	\begin{scope}
	\drawWLD{6}{1.5}
	\drawprop{2}{-1}{3}{1}
		\end{scope}
	\end{tikzpicture}

& \quad &

\begin{tikzpicture}[rotate=60,baseline=(current bounding box.east)]
	\begin{scope}
	\drawWLD{6}{1.5}
	\drawprop{1}{0}{4}{0}
	\drawprop{2}{0}{5}{0}
		\end{scope}
	\end{tikzpicture}

\end{tabular}
\end{center}
\caption{Examples of diagrams which {\em fail} to be admissible.}
\label{fig:nonadmissibleWLD}
\end{figure}
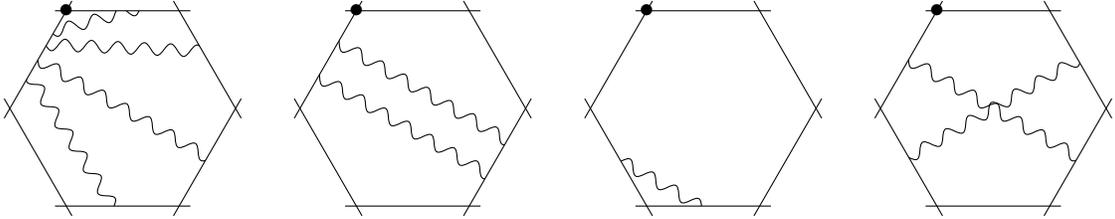

The following result makes precise the relationship between the matrices associated to admissible Wilson loop diagrams and total positivity.

\begin{thm}\label{admisWLDSarepositroids} \cite[Theorem 1.14, Theorem 3.41]{Wilsonloop}
Let $W$ be an admissible Wilson loop diagram with $k$ propagators on $n$ vertices, and $\cZ_*$ a matrix of external data. Then:
\begin{enumerate}
\item $\cC_*(W)(\cZ_*)$ is a matrix of full rank for any choice of external data $\cZ_*$, i.e. it lies in {$\mathbb{G}_{\R}(k,n+1)$}.
\item $\cC(W)$ defines a positroid (in the sense of Definition \ref{def:positroid}).
\end{enumerate}
\end{thm}

Each admissible Wilson loop diagram thus corresponds to a subspace of $\G_{\R}(k, n+1)$: the space parametrized by the matrix $\cC_*(W)$ (recall Definition \ref{Cstardef}).

Write $\cW_*(k,n) \subset \G_{\R}(k,n+1)$ for the subspace of $\G_{\R}(k,n+1)$ parametrized by all of the admissible Wilson loop diagrams. Any matrix of external data $\cZ_*$ induces a projection map from $\cW_*(k,n)$ to a subspace of $\Gr(k,k+4)$:
\bas \cZ_* : \cW_*(k,n) & \rightarrow \Gr(k, k+4) \\ \cC_*(W) & \mapsto \cC_*(W) \cdot \cZ_* \eas
The integrals associated to the admissible Wilson loop diagrams (introduced in Section \ref{sec:Feynmanrules}) define volumes on the spaces of the form $\cC_*(W)\cdot \cZ_*$, and it is these volumes which yield the scattering amplitude. The volume associated to the space $\cC_*(W) \cdot \cZ_*$ is defined by evaluating a rational function at the hyperplane that satisfies \eqref{eq:CZ = 0}.

However, in this paper we focus not on the image $\cZ_*(\cW_*(k,n))$ but on the positive geometry of the space parametrized by the Wilson loop diagrams $C(W)$. In later sections, we show that many of the problems that arise in evaluating and interpreting the integrals can be resolved by close inspection of this preimage.

\begin{rmk}
The space $\cZ_*(\cW_*(k,n))$ is conjectured to be the Amplituhedron.  See also \cite{Amplituhedronsquared} for a connection between $\cW_*(k,n)$ and a geometric space called the Amplituhedron squared.
\end{rmk}

For the rest of this paper, {\bf we consider admissible Wilson loop diagrams only}.

\subsection{Wilson loop diagrams and positroids}\label{sub:WLDpositroids}

In this paper, we are interested in the geometry defined by the admissible Wilson loop diagrams. In other words, we wish to study the subspace of $\Gr(k,n)$ parametrized by matrices of the form $\cC(W)$, making use of the CW structure coming from the positroid stratification of $\Gr(k,n)$. The underlying physical justification for taking this approach is given in Section \ref{sec:Feynmanrules}; for now, we simply note that the volume forms associated to individual  Wilson loop diagrams have singularities which are conjectured to cancel out in the final sum, and in order to verify this conjecture it suffices to study their behavior on the boundaries of the positroids appearing in Theorem \ref{admisWLDSarepositroids}.

In this section we introduce the definitions and notation which will allow us to make precise the object whose geometry we need to understand. In Sections \ref{sec:2,6case} and \ref{sec:cancellingpoles}, we verify that the singularities do indeed all cancel out in the case $k=2$, $n=6$.

\begin{dfn}
Let $W$ be an admissible Wilson loop diagram, and $\cB(W)$ the bases set of the positroid defined by $\cC(W)$. Define $\Sigma(W)$ to be the corresponding closure of the positroid cell in $\Gr(k,n)$, i.e.
\[\Sigma(W) =\overline{ \{A \in \Gr(k,n) \ |\ \Delta_I(A) \neq 0 \text{ iff } I \in \cB(W)\}}.\]
\end{dfn}

\begin{dfn} \label{dfn:WLDcomplex}
Define $\cW(k,n)$ to be the subspace of $\Gr(k,n)$ consisting of the union of the closures of the positroid cells associated to admissible diagrams, i.e.
\bas \cW(k,n) = \bigcup_{\substack{W \; admiss.\\ n\; verts, k\; props}} \overline{\Sigma(W)}. \eas
\end{dfn}

For the remainder of this paper, we restrict our attention to the geometry of $\cW(k,n)$.

It is important to note that the map from admissible Wilson loop diagrams to positroid cells is not injective.  It is certainly well-defined: by Theorem \ref{admisWLDSarepositroids}, restricting our domain to admissible diagrams ensures that we do always land in a positroid cell, but it is still possible for two different admissible Wilson loop diagrams to give rise to the same positroid cell. As such, the number of positroid cells involved in definining $\cW(k,n)$ is \emph{strictly fewer} than the number of possible admissible Wilson loop diagrams with $k$ propagators and $n$ vertices.

In order to identify precisely when this happens, we recall further notation from \cite{Wilsonloop}.

\begin{dfn} \label{exactdfn}
If $W$ is an admissible Wilson loop diagram with a non-empty set of propagators $Q \subseteq \cP$ such that $|V_Q| = |Q| +3$, then $(Q, V_Q)$ is an {\bf exact subdiagram} of $W$.
\end{dfn}

This allows us to define an equivalence relation on admissible Wilson loop diagrams.

\begin{dfn} \label{equivdfn}
Two admissible diagrams $W = (\cP, n)$ and $W'=(\cP', n)$ are equivalent if
\begin{enumerate}
\item There exists $Q \subseteq \cP \cap \cP'$ such that we can write $\cP = Q \sqcup R$, $\cP' = Q \sqcup R'$.
\item $V_R = V_{R'}$.
\item The subdiagrams $(R, V_R)$ and $(R',V_R)$ are both unions of exact subdiagrams.
\end{enumerate}
\end{dfn}

In other words, two admissible diagrams $W$ and $W'$ are equivalent if they differ only by unions of exact subdiagrams supported on the same set of vertices. See Figure \ref{fig:equiv diagrams ex} for an example of equivalent diagrams.

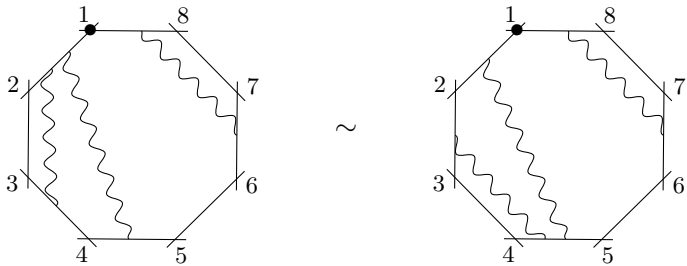
\begin{figure}[h!]
\[\begin{tikzpicture}[rotate=67,baseline=(current bounding box.east)]
	\begin{scope}
	\drawWLD{8}{1.5}
	\drawnumbers
	\drawprop{1}{1}{3}{0}
	\drawprop{1}{-1}{4}{0}
	\drawprop{6}{0}{8}{0}
		\end{scope}
	\end{tikzpicture}
\qquad \sim \qquad
\begin{tikzpicture}[rotate=67,baseline=(current bounding box.east)]
	\begin{scope}
	\drawWLD{8}{1.5}
	\drawnumbers
	\drawprop{1}{0}{4}{1}
	\drawprop{2}{0}{4}{-1}
	\drawprop{6}{0}{8}{0}
		\end{scope}
	\end{tikzpicture}
\]
\caption{An example of two equivalent diagrams with 3 propagators on 8 vertices.}
\label{fig:equiv diagrams ex}
\end{figure}

\begin{thm} \label{equivthm} \cite[Theorem 1.18]{Wilsonloop}
If $W = (\cP, n)$ and $W'=(\cP', n)$ are two  equivalent admissible Wilson  loop diagrams, then $\cC(W)$ and $\cC(W')$ define the same positroid.
\end{thm}

\section{Feynman integrals of WLDs  \label{sec:Feynmanrules}}

In this section we give more details on the physical interpretation of Wilson loop diagrams and their associated data, in order to motivate the computations of later sections.

A  holomorphic Wilson loop at $n$ points (which we do not define here) is closely related to the $n$ point scattering amplitude of supersymmetric Yang Mills Theory (SYM $N=4$) \cite{Alday:2007hr}. These Wilson loops $\mathfrak{W}_{k,n}$ are given as sums of Feynman diagrams, which in this setting are the Wilson loop {\em diagrams} defined in Section \ref{sub:WLDdef}. These diagrams represent interactions in SYM $N=4$; a diagram with $n$ external vertices and $k$ MHV propagators represents a $N^kMHV$ interaction.

Inspired by the work on the Amplituhedron, which interprets the $n$ point $N^kMHV$ on shell scattering amplitude as the volume of a ($4k$-dimensional) subspace of $\Gr(k, k+4)$ called the Amplituhedron, we study the $n$ point $N^kMHV$ scattering amplitude as a (different) volume of the subspace $\cZ_*(\cW_*(k,n)) \subset \Gr(k, k+4)$. We work under the expectation that the $\cZ_*(\cW_*(k,n))$ is closely related to the Amplituhedron. This point of view is different from work done to associate Wilson loop diagrams and the holomorphic Wilson loop to the Amplituhedron squared \cite{Amplituhedronsquared}. Namely, in this paper, we are interested in the geometry of the $3k$-dimensional subspace $\cW(k, n) \subset \Gr(k, n)$ in Definition \ref{dfn:WLDcomplex}.

Note that the amplitude itself is defined on $\cW_*(k,n)$; this is described in Sections \ref{sub:Wilsonloops} and \ref{sub:integrals} below. However, as seen in Theorem \ref{admisWLDSarepositroids}, the postive structure of the Wilson loop diagrams is encapsulated by the subspace sections $\cW(k, n) \subseteq \Gr(k,n)$. In Sections \ref{sub:mnemonic} and \ref{sec:cancellingpoles} below, we examine whether all necessary cancellations to make the theory finite can be verified by considering only the positroid tiling of $\cW(k, n)$. We verify this in the case of $\Gr(2,6)$ (Theorem \ref{boundarycancelthm}), and conjecture that it holds for general $k$ and $n$.

\subsection{Relating Wilson loops and Wilson loop diagrams}\label{sub:Wilsonloops}

The diagrams we focus on arise from a reformulation of Wilson loops in twistor space \cite{Mason:2010yk}. In particular, the equivalent object to a Feynman diagram in this setting is the Wilson loop diagram (WLD). Recall from Section \ref{sub:notation} that the vertices of the external convex polygon of a WLD correspond to the external particles of an interaction, which we represent as a matrix $\cZ \in M_{\R, +} (n, k+4)$ with the $i^{th}$ row (denoted $Z_i$) corresponding to the $i^{th}$ vertex. We also consider one additional vector $Z_* \in \R^{4+k}$, which represents a choice of gauge.

Each $Z_i$ is a section of a $k$-dimensional real vector bundle over twistor space; the first 4 components $Z_i^\mu$ corresponding to the real momentum data in twistor space, and the last $k$ components corresponding to the bosonized fermionic data. Without loss of generality, we may choose $Z_*$ to be the $0$ section in this bundle., i.e. such that the final $k$ entries of $Z_*$ are $0$.

Each propagator depicted in the Wilson loop diagram corresponds to an MHV propagator of the overall interaction. If the propagator $p$ corresponds to the $b^{th}$ row of the matrix $\cC_*(W)$, then it is represented by the vector \ba Y_b = \cC_*(W)_b \cdot \cZ_* \in \R^{4+k} \; . \label{Ybeq}\ea

Just as a scattering amplitude is given as a sum of Feynman integrals, the holomorphic Wilson loop $\mathfrak{W}_{k,n}(\cZ)$ is given as a sum of certain integrals $\cI(W)(\cZ_*)$ associated to Wilson loop diagrams $W$ (see equation \eqref{Iwzeq} below). Given a matrix of external data $\cZ_*$ satisfying the conditions of Definition \ref{def:cZ def}, the integral $I(W)(\cZ_*)$ assigns a volume to the space $\cC_*(W) \cdot \cZ_*$. The Wilson loop is then given by
\ba \mathfrak{W}_{k,n}(\cZ) = \sum_{\begin{subarray}{c}W \textrm{admis.} \\ n \textrm{ point N${^k}$MHV diag.}\end{subarray}} \cI(W)(\cZ_*)\;. \label{holowilsonloop}\ea
For more on the relation between Wilson loop diagrams and traditional Feynman diagrams, see \cite{Adamo:2011cb}. In particular, the diagrams we consider in this paper correspond to to tree level amplitudes (no internal loops). The analysis in this paper can be extended to loop level, but we restrict ourselves to tree level interactions here.

There is one important subtlety to interpreting the holomorphic Wilson loop geometrically as in \eqref{holowilsonloop}. If two Wilson loop diagrams are equivalent, then they define the same positroid cell, i.e. $\Sigma(W) = \Sigma(W')$.  The spaces $\cC_*(W)$ and $\cC_*(W')$ are therefore also equal. If one were to rewrite the expression in \eqref{holowilsonloop} as a sum over distinct subspaces of $\cW_*(k,n)$, one would need to sum over volume functionals associated to each subspace: \ba \textrm{vol } \cC_*(W)  = \sum_{\substack{W' \textrm{ s.t. } \\ W' \sim W}} \cI(W') \;. \label{cellvolumeeq} \ea
Combining \eqref{holowilsonloop} and \eqref{cellvolumeeq} gives
\bas \mathfrak{W}_{k,n} = \sum_{\cC_*(W) \subset \cW_*(k,n)} \ \ \sum_{\substack{W' \textrm{ s.t. } \\ W' \sim W}} \cI(W') \;.\eas
Thus in order to study the Wilson loop $\mathfrak{W}_{k,n}$, we need to know exactly which positroids lie in $\cW(k,n)$ (the outer summation), how many equivalent diagrams are associated to each positroid (the inner summation), and of course how to compute the integral $\cI(W)$ for each WLD $W$.

\subsection{From the Wilson loop diagram $W$ to the integral $\cI(W)$}\label{sub:integrals}

Let $W = (\cP, n)$ be an admissible Wilson loop diagram with $|\cP| = k $ propagators, and recall the notation for the external data $\cZ_*$ given in Section \ref{sub:notation} above. The integral $\cI(W)$ is a functional associated to the space parametrized by $\cC_*(W)$: evaluating it on a choice of data $\cZ_*$ yields (a component of) the Wilson loop for that particular external data. We define $\cI(W)$ as a function of $\cZ_*$ as follows \cite{Adamo:2011cb,Mason:2010yk}:
\ba \cI(W)(\cZ_*) = \int_{\R^{4k}} \frac{\prod_{b= 1}^k \prod_{a \in V_{p_b}} \rd c_{b,a}}{R(W)} \delta^{4k | 4k}(\cC_*(W) \cdot \cZ_*) \;, \label{Iwzeq}\ea
where the $c_{b,a}$ are the entries of the matrix $\cC(W)$, and $\delta^{4k|4k}$ and $R(W)$ are given in Definitions \ref{def:delta} and \ref{def:RW} respectively. An example of the computation of $\cI(W)(\cZ_*)$ is given in Example \ref{complicateddiagrameg} below.

\begin{dfn}\label{def:delta}The notation $\delta^{4k|4k}$ is related to the Dirac delta function defined on a vector with both bosonic and fermion components. Explicitly, we have
\ba \delta^{4k | 4k}(\cC_*(W) \cdot \cZ_*) = \prod_{b=1}^k (Y_b^{4 + b})^4 \delta^{4}(Y_b^\mu) \;, \label{deltanneq}\ea
where $Y_b$ are the vectors defined in \eqref{Ybeq}, $Y_b^{4+b}$ is the ($4+b$)th entry of $Y_b$, and $Y_b^{\mu}$ is the projection of $Y_b$ to its first four entries (as described in Section \ref{sub:notation}).
\end{dfn}

In order to define the denominator $R(W)$, recall that the vertex labelling of $W$ induces a labelling of the edges, where edge $i$ connects vertex $i$ and $i+1$. For each propagator $p = (i,j)$ in $\cP$, define its edge support set to be $E_p= \{i, j\}$. As with the vertex sets, we extend this to any subset $Q \subseteq \cP$ by setting $E_Q = \cup_{p\in Q} E_p$.

\begin{dfn}\label{def:RW}  If $\{q_1 \ldots q_s\}$ is the set of propagators incident to the edge $i$, ordered such that $q_1$ is the counterclockwise most (closest to the vertex $i$) and $q_s$ the clockwise most (closest to the vertex $i+1$), define
\bas R_{i}(W) =   c_{q_1, i+1}c_{q_s, i} \prod_{m=1}^{s-1}(c_{q_m, i}c_{q_{m+1}, i+1}- c_{q_{m+1}, i} c_{q_m, i+1}) \;.\eas
Note that if $i$ only supports one propagator, then $s=1$ and $R_{i}(W) = c_{p, i} c_{p, i+1 } $.  The denominator $R(W)$ is then defined by
\bas R(W) = \prod_{e \in E_{\cP}} R_e(W) \;.\eas
\end{dfn}

Evaluating $\cI(W)$ corresponds to performing the Dirac delta functions $\delta^{4}(Y_b^\mu)$ and evaluating the expression $\frac{\prod_{b=1}^k(Y_b^{4 + b})^4}{R(W)}$ at the corresponding points. By \eqref{deltanneq} this means we should evaluate it at the solution of the system of equations \bas \cC_*(W) \cdot \cZ_*^\mu  = \mathbf{0} \; .\eas  (This process is sometimes called {\em localizing} the integral at the delta function.)

Since we have an explicit description for the hyperplane on which this occurs (recall equation \eqref{Cmatrixentries}) we can compute the integral \eqref{Iwzeq} in terms of $\cZ_*$ as follows.

We first consider the denominator, writing $R(W)(\cZ_*)$ to denote the localization of $R(W)$ at a given choice of external data $\cZ_*$. For each edge $i$ of $W$, we have \bas R_{i}(W)(\cZ_*) =   \frac{\sigma_{q_1, i+1}\sigma_{q_s, i} \prod_{m=1}^{s-1}(\sigma_{q_m, i}\sigma_{q_{m+1}, i+1}- \sigma_{q_{m+1}, i} \sigma_{q_m, i+1})}{\prod_{m=1}^k \langle i_{q_m} i_{q_{m+1}} j_{q_m} j_{q_{m+1}}\rangle^2 }\;,\eas where $\sigma_{b,a}$ is as in Definition \ref{sigmadfn}. Combining this for all edges as above, we obtain

\[R(W)(\cZ_*) = \prod_{e\in E_{\cP}}  R_{e}(W)(\cZ_*).\]

Note that $R(W)(\cZ_*)$ is a degree $0$ rational function in the determinants $\sigma_{b,a}$ and $\langle i_{p_b}i_{p_b+1}j_{p_b}j_{p_b+1}\rangle$. The {\bf physical} singularities of the theory occur when \ba \langle i_{p_b}i_{p_b+1}j_{p_b}j_{p_b+1}\rangle = 0 \label{physsingeq} \;.\ea The simple poles of $\cI(W)(\cZ_*)$ that occur when $\sigma_{b,a} = 0$ or $\sigma_{b, a}\sigma_{c, a+1} - \sigma_{c,a}\sigma_{b, a+1}=0$ are the {\bf spurious} singularites of the theory, and are expected to cancel in the sum in equation \eqref{holowilsonloop}.  There are circumstances under which the various factors of $R(W)(\cZ_*)$ are not distinct, in which case $I(W)(\cZ_*)$ has poles of higher degree. In this paper, we only consider the simple poles; see Remark \ref{rmk:bdnyspoles}.

In order to examine the numerator of $\cI(W)(\cZ_*)$, recall from \eqref{Ybeq} that given a Wilson loop diagram $W = (\cP, n)$ with $p_b = (i,j) \in \cP$ corresponding to $b^{th}$ row of $\cC_*(W)$, we have
\bas Y_p^b = c_{b,i}Z_{i}^{4+b} + c_{b,i+1}Z_{i+1}^{4+b} + c_{b,j}Z_{j}^{4+b} + c_{b,j+1}Z_{j+1}^{4+b}\;.\eas
We note that because of the symmetries in the bosonization process, the integral $\cI(W)$ is invariant the symmetric group $S_k$ acting on the rows of $\cC_*(W)$.

After localization, we obtain \ba F_p^b := (Y_p^b)^4 = \frac{1}{\langle i (i+1) j (j+1) \rangle^4}(\sigma_{b,i}Z_{i}^{4+b} + \sigma_{b,i+1}Z_{i+1}^{4+b} + \sigma_{b,j}Z_{j}^{4+b} + \sigma_{b,j+1}Z_{j+1}^{4+b})^4\;.\label{Feq}\ea

In other words, the integral $\cI(W)$ evaluates to \ba \cI(W)(\cZ_*) = \frac{\prod_{b=1}^k F_p^b}{R(W)(\cZ_*)} \label{localizedint}\;.\ea

Since SYM $N=4$ is a finite theory \cite{Alday:2007hr, ParkeTaylorids}, the scattering amplitudes (and thus the holomorphic Wilson loops $\,\mathfrak{W}_{k,n}$) are finite. However, as seen above, the integrals $\cI(W)(\cZ_*)$ have spurious poles. In order for these poles to cancel, they must appear on the boundaries of the cells $\Sigma(W)$ and cancel exactly in the induced tiling. This is parallel to the cancellation of poles associated to the BCFW integrals in the Amplituhedron calculation \cite{Arkani-Hamed:2013jha}. In Section \ref{sec:cancellingpoles} we explicitly show this cancellation for the case $k=2$, $n=6$.

\subsection{Example: computing $\cI(W)$ for a Wilson loop diagram with $k=3$, $n=8$}

Before proceeding with the more geometric aspects of these Wilson loop diagrams, we give an example of the integrals and rational functions involved. As we only consider the case of $N^2MHV$ diagrams, this is a more complicated example than any we consider in the paper, but we include it for give a fuller flavour of the calculations involved.

\begin{eg} \label{complicateddiagrameg}
Consider the following diagram:
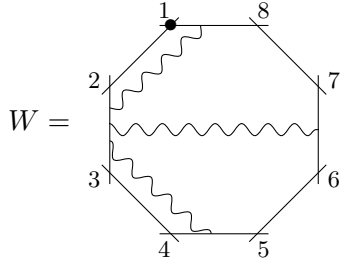
\bas W = {\begin{tikzpicture}[rotate=67.5,baseline=(current bounding box.east)] 
\begin{scope}
\drawWLD{8}{1.5} 
\drawnumbers 
\drawprop{2}{-1}{8}{0} 
\drawprop{2}{0}{6}{0}
\drawprop{2}{1}{4}{0}
\end{scope}
\end{tikzpicture} }
\eas

From Definition \ref{Cstardef}, we have \bas C_*(W) = \begin{bmatrix}1 & c_{1,1} & c_{1,2} & c_{1,3} & 0 & 0 & 0 & 0& c_{1,8} \\ 1 & 0 & c_{2,2} & c_{2,3} & 0 & 0 & c_{2,6} & c_{2,7} & 0 \\ 1 & 0 & c_{3,2} & c_{3,3} & c_{3,4} & c_{3,5} & 0 & 0& 0\end{bmatrix} \;. \eas

Then by the algorithm in Definition \ref{def:RW}, we obtain \bas R(W) = c_{1,3}(c_{1,2}c_{2,3} - c_{2,2}c_{1,3})(c_{2,2}c_{3,3} - c_{3,2}c_{2,3})c_{3,2} c_{3,4}c_{3,5} c_{2,6}c_{2,7}  c_{1,1}c_{1,8} \; .\eas

Localizing the integral against the product of delta functions $\prod_{b=1}^3\delta^4(Y^\mu_b)$ as in Definition \ref{def:delta} gives

\bas \cC_*(W)(\cZ) = \begin{bmatrix} 1 & \frac{\langle 238*\rangle}{\langle 2381\rangle} & \frac{\langle *381\rangle}{\langle 2381\rangle} & \frac{\langle 2*81\rangle}{\langle 2381\rangle}& 0 & 0 & 0 & 0 & \frac{\langle 23*1\rangle}{\langle 2381\rangle} \\1 & 0 & \frac{\langle *367\rangle}{\langle 2367\rangle} & \frac{\langle 2*67\rangle}{\langle 2367\rangle}& 0 & 0 & \frac{\langle 23*7\rangle}{\langle 2367\rangle} & \frac{\langle 236*\rangle}{\langle 2367\rangle} & 0 \\ 1 & 0 & \frac{\langle *345\rangle}{\langle 2345\rangle} & \frac{\langle 2*45\rangle}{\langle 2345\rangle}& \frac{\langle 23*5\rangle}{\langle 2345\rangle} & \frac{\langle 234*\rangle}{\langle 2345\rangle} & 0& 0& 0  \end{bmatrix}.
\eas

By choosing a $Z_*$ with 0 fermionic components (i.e. with $0$ in the final $k$ entries) and assuming that the vectors $Z_i^\mu$ are in general position (i.e. that the determinants not involving $Z_*^\mu$ are all non-zero), we may evaluate this integral to be

\begin{align*}\cI(W) (\cZ_*) &=  \frac{\left(\begin{multlined}\big(\langle *381\rangle Z_2^5 + \langle 2\!*\!81\rangle Z_3^5 + \langle 23\!*\!1\rangle Z_8^5 + \langle 238*\rangle Z_1^5\big)^4
\\  \cdot \big(\langle *367\rangle Z_2^6 + \langle 2\!*\!67\rangle Z_3^6 + \langle 23\!*\!7\rangle Z_6^6 + \langle 236*\rangle Z_7^6\big)^4
\\ \cdot \big(\langle *345\rangle Z_2^7 + \langle 2\!*\!45\rangle Z_3^7 + \langle 23\!*\!5\rangle Z_4^7 + \langle 234*\rangle Z_5^7\big)^4 \end{multlined}\right)}{\left(\begin{multlined}\langle 2\!*\!81\rangle \big(\langle *281\rangle \langle 2\!*\!67\rangle  - \langle *367\rangle \langle 2\!*\!81\rangle\big)\big(\langle *367\rangle  \langle 2\!*\!45\rangle  - \langle *345\rangle \langle 2\!*\!67\rangle \big)\langle *345\rangle \\ \cdot \langle 23\!*\!5\rangle \langle 234*\rangle \langle 23\!*\!7\rangle \langle 236*\rangle \langle 238*\rangle \langle 23\!*\!1\rangle
  \end{multlined}\right)}
 \\[1em]
& =
\frac{\left(\begin{multlined}\big(\sigma_{1,2} Z_2^5 + \sigma_{1,3} Z_3^5 + \sigma_{1,8} Z_8^5 + \sigma_{1,1} Z_1^5\big)^4\\\cdot\big(\sigma_{2,2} Z_2^6 + \sigma_{2,3} Z_3^6 + \sigma_{2,6} Z_6^6 + \sigma_{2,7} Z_7^6\big)^4\\\cdot\big(\sigma_{3,2} Z_2^7 + \sigma_{3,3} Z_3^7 + \sigma_{3,4} Z_4^7 + \sigma_{3,5} Z_5^7\big)^4\end{multlined}\right)}
{\sigma_{1,3}\, ( \sigma_{1,2} \sigma_{2,3} -  \sigma_{2,2} \sigma_{1,3})( \sigma_{2,2} \sigma_{3,3} -  \sigma_{3,2} \sigma_{2,3}) \, \sigma_{3,2} \, \sigma_{3,4}\,  \sigma_{3,5} \, \sigma_{2,6}\, \sigma_{2,7} \, \sigma_{1,1}\, \sigma_{1,8}     }.\end{align*}

Notice that this integral has only $10$ spurious poles singularities, two of which are defined by the more complicated expressions $( \sigma_{1,2} \sigma_{2,3} -  \sigma_{2,2} \sigma_{1,3})$ and $( \sigma_{2,2} \sigma_{3,3} -  \sigma_{3,2} \sigma_{2,3})$. If we had instead considered a diagram that did not have propagators sharing a terminal edge, we would have found $12$ (i.e. $4k$) spurious poles, each given by a single $\sigma_{b,a}$.
\end{eg}

\section{The geometry of Wilson loop diagrams representing $N^2MHV$ diagrams at $6$ points}\label{sec:2,6case}

We now restrict our attention to the space $\cW(2,6)$, which is tiled by cells associated to admissible Wilson loop diagrams with 6 vertices and 2 propagators. The aim of this section is to examine the $\cW(2,6)$ case in detail, computing the Le diagrams associated to each WLD, their codimension 1 boundaries, and the homology of the subcomplex in $\Gr(2,6)$ whose top-dimensional cells are precisely those in $\cW(2,6)$.

By computing all possible codimension 1 boundaries shared by pairs of cells in $\cW(2,6)$, we also show that the standard diagramatic tool used by physicists to identify shared boundaries of Wilson loop diagrams is insufficient even in this small case, as it does not ``see'' many of the shared boundaries. This suggests that the combinatorial approach to studying Wilson loop diagrams is a fruitful one.

For simplicity and clarity, we omit the computations underlying the data in this section, but note that they can easily be reconstructed using the vertex-disjoint path system approach described in Section \ref{sub:positroidbackground}.\footnote{The Python code used by the authors to perform these computations is available as an auxiliary file to this paper on arXiv.} Many other tools are available to study the general case, which we explore further in a forthcoming paper.

\subsection{The Le diagrams associated to Wilson loop diagrams in $\cW(2,6)$}\label{sub:2,6data}

In Table \ref{nameWLDLetable} we list the 21 admissible Wilson loop diagrams with 2 propagators on 6 vertices, along with the Le diagram of their associated positroid. We also give each WLD a name ($V_{\bullet}$, $P_{\bullet}$, or $E_{\bullet}$) in order to easily refer to them later. Notice that the rotational symmetry of the Wilson loop diagrams is reflected in the Le diagrams.

\begin{longtable}[t]{ccl || ccccl }

\multicolumn{2}{c}{\textbf{WLD}} & \textbf{Le diagram} & \multicolumn{4}{c}{\textbf{WLD}} & \textbf{Le diagram} \\
$V_1$ & \makediag{1}{-1}{5}{0}{1}{1}{3}{0}& \ytableaushort{0+0+,++++} &  $E_{1,L}$  & \makediag{5}{1}{2}{0}{5}{-1}{3}{0}& $E_{1,R}$ & \makediag{2}{-1}{5}{0}{2}{1}{4}{0} & \ytableaushort{+++,+++}\\
$V_2$ & \makediag{6}{0}{2}{-1}{2}{1}{4}{0} &\ytableaushort{+0++,0+++}&  $E_{2, L}$ & \makediag{6}{1}{3}{0}{6}{-1}{4}{0} & $E_{2, R}$ &\makediag{3}{-1}{6}{0}{3}{1}{5}{0}&\ytableaushort{+++0,+++} \\
$V_3$ & \makediag{1}{0}{3}{-1}{3}{1}{5}{0} & \ytableaushort{0+++,+++} &  $E_{3,L}$ &\makediag{1}{1}{4}{0}{1}{-1}{5}{0} & $E_{3,R}$ &\makediag{4}{-1}{1}{0}{4}{1}{6}{0}&\ytableaushort{+++0,+++0} \\
$V_4$ & \makediag{6}{0}{4}{1}{2}{0}{4}{-1}&\ytableaushort{+++0,0+++} &  $E_{4, L}$ & \makediag{2}{1}{5}{0}{2}{-1}{6}{0} & $E_{4, R}$ & \makediag{5}{-1}{2}{0}{5}{1}{1}{0}&\ytableaushort{++0+,++0+}\\
$V_5$ &\makediag{1}{0}{5}{1}{3}{0}{5}{-1}&\ytableaushort{++0+,+++} &  $E_{5,L}$ & \makediag{3}{1}{6}{0}{3}{-1}{1}{0} & $E_{5,R}$ & \makediag{6}{-1}{3}{0}{6}{1}{2}{0} & \ytableaushort{+0++,+0++} \\
$V_6$ & \makediag{6}{1}{2}{0}{6}{-1}{4}{0} &\ytableaushort{+0+0,++++} &  $E_{6, L}$ & \makediag{4}{1}{1}{0}{4}{-1}{2}{0} & $E_{6, R}$&\makediag{1}{-1}{4}{0}{1}{1}{3}{0} &\ytableaushort{0+++,0+++}\\

$P_1$ &\makediag{1}{0}{3}{0}{4}{0}{6}{0}&\ytableaushort{0++0,++++}\\
$P_2$ & \makediag{1}{0}{5}{0}{2}{0}{4}{0}&\ytableaushort{++0+,0+++}\\
$P_3$ &\makediag{6}{0}{2}{0}{3}{0}{5}{0}&\ytableaushort{+0++,+++}\\

\caption{All Wilson loop diagrams with $k=2$ and $n=6$, and their associated Le diagrams.} \label{nameWLDLetable}
\end{longtable}

Recall from \eqref{cellvolumeeq} that \bas \textrm{vol}_{\cZ_*} (\cC_*(W) \cdot \cZ_*) = \sum_{W' \sim W} I(W')(\cZ_*)\;.\eas That is, the volume of the space parametrized by matrices $\cC_*(W)$ of a Wilson Loop diagram $W$ is given by a sum of integrals associated to all diagrams equivalent to $W$ (including $W$ itself). If $W$ contains an exact subdiagram, then this equivalence class may contain more than one diagram, and thus the volume may involve more than one integral.

\begin{eg}
From Table \ref{nameWLDLetable}, we see that there is an equivalence between each pair of Wilson loop diagrams $E_{i,L}$ and $E_{i,R}$. Indeed, $E_{i,L}$ and $E_{i,R}$ both have two propagators, supported in each case on the set $V_{\cP} = [6]\setminus \{i\}$, and these propagators form an exact subdiagram (see Definition \ref{exactdfn}). By Definition \ref{equivdfn} we therefore have $E_{i,L} \sim E_{i,R}$, and hence by Theorem \ref{equivthm} they have the same associated positroid.

In the case of $N^2MHV$ diagrams at 6 points, these are the only non trivial equivalence classes. These are also the only diagrams that contain exact subdiagrams.
\end{eg}

\begin{rmk}
There are also six $6$-dimensional positroid cells in $\Gr(2,6)$ that do not correspond to any Wilson loop diagram. They are listed in Figure \ref{missingcellsfigure}.
\begin{figure}[h!]
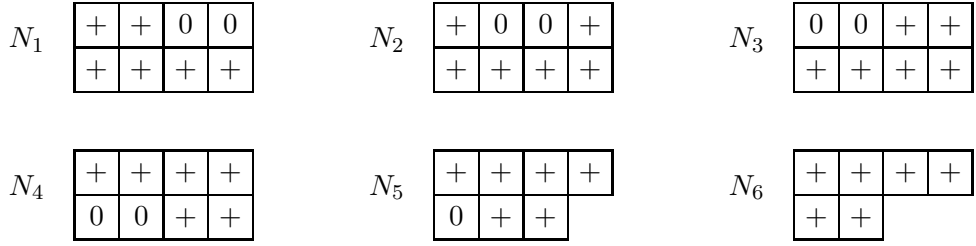

\begin{gather*} N_1 \quad \ytableaushort{++00,++++} \qquad \qquad N_2 \quad \ytableaushort{+00+,++++} \qquad \qquad N_3 \quad \ytableaushort{00++,++++}\\
\
\\ N_4\quad \ytableaushort{++++,00++} \qquad \qquad N_5 \quad \ytableaushort{++++, 0++} \qquad \qquad N_6 \quad \ytableaushort{++++,++} \end{gather*}
\caption{The 6-dimensional positroid cells in $\Gr(2,6)$ that do not correspond to Wilson loop diagrams. As in Table \ref{nameWLDLetable}, we assign labels to each cell in order to refer to them later.}
\label{missingcellsfigure}
\end{figure}
\end{rmk}

\subsection{The geometry of $\cW(2,6)$}\label{sub:2,6geometry}

The spurious singularities of the integrals $\cI(W)(\cZ_*)$ in equation \eqref{Iwzeq} are conjectured to cancel on the codimension 1 boundaries shared between pairs of cells associated to Wilson loop diagrams. In this section, we use the technology described in Section \ref{sec:WLDbackground} to establish exactly which cells in $\cW(2,6)$ share codimension 1 boundaries in the CW complex of $\Gr(2,6)$.

In this small and easily computable setting, identifying a shared boundary proceeds as follows. Given two Le diagrams $L_1$ and $L_2$, each representing positroid cells of dimension $d$, we can first obtain their bases sets $\BB_1$ and $\BB_2$ via the method of vertex-disjoint path systems. Then the cells corresponding to $L_1$ and $L_2$ share a $d-1$ dimension boundary if and only if there is a Le diagram $L_3$ with bases set $\BB_3$ such that \begin{enumerate} \item $L_3$ has exactly $d-1$ squares containing a $+$ symbol; \item $\BB_3 \subset \BB_2 \cap \BB_1$. \end{enumerate}

We now give several examples to illustrate this method, and to highlight the types of behavior exhibited by the cells in $\cW(2,6)$.

\begin{eg}\label{eg:bounds1}
Consider the Wilson loop diagrams $V_1$ and $E_{6,R}$, and their associated positroid cells:

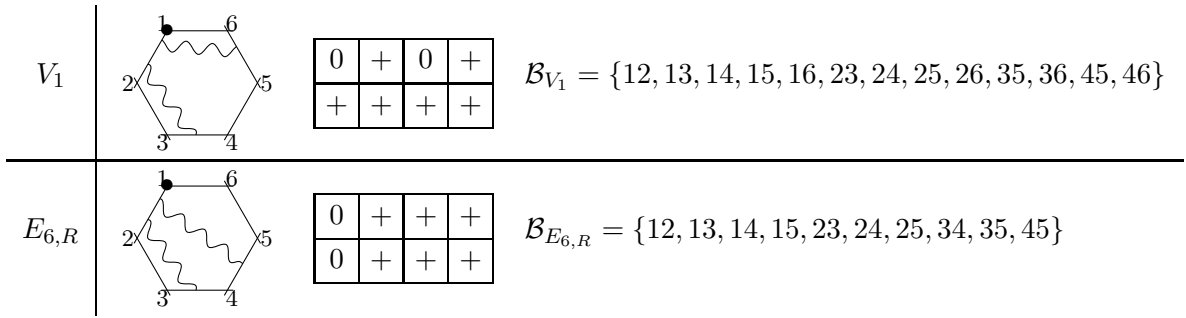
\[\begin{tabular}{c|ccl}
$V_1$ & \begin{tikzpicture}[rotate=60,baseline=(current bounding box.east)]
	\begin{scope}
	\drawWLD{6}{0.8}
	\drawnumbers
	\drawprop{1}{-1}{5}{0}
	\drawprop{1}{1}{3}{0}
		\end{scope}
	\end{tikzpicture}

 & \ytableaushort{0+0+,++++} &
$\BB_{V_1} = \{12, 13, 14, 15, 16, 23, 24, 25, 26, 35, 36, 45, 46 \}$
\\ \hline
$E_{6,R}$
&
\begin{tikzpicture}[rotate=60,baseline=(current bounding box.east)]
	\begin{scope}
	\drawWLD{6}{0.8}
	\drawnumbers
	\drawprop{1}{-1}{4}{0}
	\drawprop{1}{1}{3}{0}
		\end{scope}
	\end{tikzpicture}
&
\ytableaushort{0+++,0+++}
&
$\BB_{E_{6,R}} = \{12,13,14,15,23,24,25,34,35,45\}$
\\

\end{tabular}\]

Drawing on the intuition of vertex-disjoint path systems, it is clear that adding more zeros to a Le diagram (while continuing to respect the Le condition) corresponds to constructing a positroid cell lying on the boundary of the original cell. With this in mind, there is an obvious choice for a codimension 1 boundary shared by $V_1$ and $E_{6,R}$, namely
\[\ytableaushort{0+0+,0+++}.\]
This Le diagram has the basis set $\{12,13,14,15,23,24,25,35,45\}$, which is precisely the intersection of $\BB_{V_1}$ and $\BB_{E_{6,R}}$.
\end{eg}

\begin{eg}\label{eg:bounds2}
Now consider the Wilson loop diagrams $V_1$ and $P_1$:

\[\begin{tabular}{c|ccl}
$V_1$ & \begin{tikzpicture}[rotate=60,baseline=(current bounding box.east)]
	\begin{scope}
	\drawWLD{6}{0.8}
	\drawnumbers
	\drawprop{1}{-1}{5}{0}
	\drawprop{1}{1}{3}{0}
		\end{scope}
	\end{tikzpicture}

 & \ytableaushort{0+0+,++++} &
$\BB_{V_1} = \{12, 13, 14, 15, 16, 23, 24, 25, 26, 35, 36, 45, 46 \}$
\\ \hline
$P_1$
&
\begin{tikzpicture}[rotate=60,baseline=(current bounding box.east)]
	\begin{scope}
	\drawWLD{6}{0.8}
	\drawnumbers
	\drawprop{6}{0}{4}{0}
	\drawprop{1}{0}{3}{0}
		\end{scope}
	\end{tikzpicture}
&
\ytableaushort{0++0, ++++}
&
$\BB_{P_1} = \{12, 13, 14, 15, 16, 24, 25, 26, 34, 35, 36, 45, 46\}$
\\

\end{tabular}\]

Once again there is an ``obvious'' codimension 1 boundary:
\[\ytableaushort{0+00,++++},\]
with bases $\{12,13,14,15,16,25,26,35,36,45,46\} \subsetneq \BB_{V_1} \cap \BB_{P_1}$.
In fact, $V_1$ and $P_1$ share a second codimension 1 boundary, namely
\[\ytableaushort{0++0,+++0},\] corresponding to the bases set $\{12,14,15,16,24,25,26,35,36,45,46\}$.
This is a manifestation of the same type of behavior highlighted in Example \ref{eg:badLebehavior}, and demonstrates why a simpler approach to identifying lower dimensional boundary cells (e.g. via propagator moves, as in Section \ref{sub:mnemonic} below, or by constructing Le diagrams ``by eye'' as in Example \ref{eg:bounds1}) is insufficient.

Having identified this unexpected second boundary, it is now easy to see (using Table \ref{nameWLDLetable}) that this boundary is also shared by the cell $\Sigma(E_{3,\bullet})$.
\end{eg}

\begin{eg}\label{eg:bounds3}
The 5-dimesional positroid cell corresponding to the Le diagram
\[\ytableaushort{000+,++++}\]
also lies on the boundary of $\Sigma(V_1)$ in $\Gr(2,6)$. However, this $5$-dimensional cell does not lie on the boundary of any of the other positroid cells listed in Table \ref{nameWLDLetable}; instead, it shares a boundary (in $\Gr(2,6)$) with the 6-dimensional cells
\[N_2 = \ytableaushort{+00+,++++} \quad \text{ and } \quad N_3 = \ytableaushort{00++,++++}.\]
As per Figure \ref{missingcellsfigure}, neither of these correspond to Wilson loop diagrams.
\end{eg}

Therefore, in Examples \ref{eg:bounds1} through \ref{eg:bounds3}, we see that the behavior of boundaries between positroid cells in $\Gr(2,6)$ is quite complicated. Certain cells, such as $\Sigma(E_{6,R})$ and $\Sigma(V_1)$ share one codimension $1$ boundary between them. Others, such as  $\Sigma(V_1)$ and $\Sigma(P_1)$ share two codimension one boundaries between them. Finally, there are sets of 3 cells that share codimension $1$ boundaries with each other, for instance $\Sigma(V_1)$, $\Sigma(N_2)$ and $\Sigma(N_3)$.

In $\Gr(2,6)$ there are 21 positroid cells of dimension 6 and 50 of dimension 5. From Table \ref{nameWLDLetable} we know that 15 of the 6-dimensional cells in $\Gr(2,6)$ appear in $\cW(2,6)$. Furthermore, from direct computation we find that $38$ of the 5-dimensional cells of $\Gr(2,6)$ share a boundary with at least two distinct positroid cells associated to Wilson loop diagrams. A further six of the 5-dimensional cells are boundaries of cells $\Sigma(E_{i})$, and are thus associated to two different Wilson loop diagrams. The remaining six 5-dimensional cells in $\Gr(2,6)$ are each boundaries of precisely one WLD.

\begin{figure}[p!]
\begin{center}

\begin{tikzpicture}
\pgfmathsetmacro{\radiusA}{1}
\pgfmathsetmacro{\radiusB}{2.5}
\pgfmathsetmacro{\angle}{360/8}

\node (M) at (0,0) {$V_i$};

\foreach \i in {0,...,7} \draw let \n1={\i} in (\angle*\n1+ 90:\radiusA) node(B\n1){$\cdot$};

\foreach \i in {0,1,2,3,5,6,7} \draw let \n1={\i} in (M) -- (B\n1);

\draw[gray,dashed] (M) -- (B4);

\node (1) at (90 + \angle/4:\radiusB) {$V_{i+2}$};
\node (2) at (90 - \angle/4:\radiusB) {$V_{i+4}$};

\draw (B0) -- (1) (B0) --(2);

\node (3A)[gray] at (90 + \angle :\radiusB) {$N_{i+3}$};
\node (3) at (90 + \angle + \angle/2:\radiusB) {$P_{i+1}$};
\node (4) at (90 + 2*\angle:\radiusB) {$E_{i+5}$};

\draw[gray] (B1) -- (3A);
\draw (B1) -- (3) (B2) -- (3);
\draw (B2)-- (4);

\node (5) at (90 + 3*\angle:\radiusB) {$E_{i+4}$};
\node(5A)[gray] at (90 + 4*\angle - \angle/4:\radiusB) {$N_{i+1}$};
\node(5B)[gray] at (90 +4*\angle + \angle/4:\radiusB) {$N_{i+2}$};
\node (6) at (90 + 5*\angle:\radiusB) {$E_{i+3}$};

\draw (B3) -- (5) (B5) -- (6);
\draw[gray]  (B3) -- (5A) (B5) -- (5B);
\draw[gray,dashed] (B4) -- (5B) (B4) -- (5A);

\node (7) at (90 + 6*\angle:\radiusB) {$E_{i+2}$};
\node (8) at (90 + 6*\angle + \angle/2:\radiusB) {$P_{i}$};

\node (8A)[gray] at (90 + 7*\angle:\radiusB) {$N_i$};

\draw (B6) -- (7);
\draw (B6) -- (8) (B7) -- (8);
\draw[gray] (B7) -- (8A);

\end{tikzpicture}

\vspace{2em}

\begin{tikzpicture}
\pgfmathsetmacro{\radiusA}{1}
\pgfmathsetmacro{\radiusB}{2.5}
\pgfmathsetmacro{\angleA}{360/4}
\pgfmathsetmacro{\angleB}{\angleA/5}

\node (M) at (0,0) {$P_i$};

\foreach \i in {0,...,3} \draw let \n1={\i} in (\angleA*\n1+ 135 - \angleB:\radiusA) node(B\n1){$\cdot$};
\foreach \i in {0,...,3} \draw let \n1={\i} in (\angleA*\n1+ 135 + \angleB:\radiusA) node(C\n1){$\cdot$};

\foreach \i in {0,...,3} \draw let \n1={\i} in (M) -- (B\n1);
\foreach \i in {0,...,3} \draw let \n1={\i} in (M) -- (C\n1);

\node (0) at (135 + 0*\angleA:\radiusB) {$V_{i+5}$};
\node (1) at (135 + 1*\angleA:\radiusB) {$V_{i+3}$};
\node (2) at (135 + 2*\angleA:\radiusB) {$V_{i+2}$};
\node (3) at (135 + 3*\angleA:\radiusB) {$V_{i}$};

\foreach \i in {0,...,3} \draw let \n1={\i} in (\n1) -- (B\n1) (\n1) -- (C\n1);

\node (00) at (135 + 0.3*\angleA:\radiusB) {$E_{i+4}$};
\node (01) at (135 + 1.3*\angleA:\radiusB) {$E_{i+5}$};
\node (02) at (135 + 2.3*\angleA:\radiusB) {$E_{i+1}$};
\node (03) at (135 + 3.3*\angleA:\radiusB) {$E_{i+2}$};

\foreach \i in {0,...,3} \draw let \n1={\i} in (0\n1) -- (C\n1) ;

\node (000)[gray] at (135 - 0.3*\angleA:\radiusB) {$N_{i+2}$};
\node (001)[gray] at (135 - 3.3*\angleA:\radiusB) {$N_{i+3}$};
\node (002)[gray] at (135 - 2.3*\angleA:\radiusB) {$N_{i+5}$};
\node (003)[gray] at (135 - 1.3*\angleA:\radiusB) {$N_{i}$};

\foreach \i in {0,...,3} \draw[gray] let \n1={\i} in (00\n1) -- (B\n1) ;

\end{tikzpicture}

\vspace{2em}

\begin{tikzpicture}[rotate=-90]

\pgfmathsetmacro{\radiusA}{1}
\pgfmathsetmacro{\radiusB}{2}
\pgfmathsetmacro{\angleA}{360/5}

\node (M) at (0,0) {$E_i$};

\node (B0) at (0:\radiusA) {$*$};
\foreach \i in {1,...,4} \draw let \n1={\i} in (\angleA*\n1:\radiusA) node(B\n1){$\cdot$};
\foreach \i in {0,...,4} \draw let \n1={\i} in (M) -- (B\n1);

\node (0)[gray] at (0*\angleA:\radiusB) {$N_{i+4}$};
\node (1) at (1*\angleA  - \angleA/5:\radiusB)  {$P_{i+2}$};
\node (2) at (1*\angleA  + \angleA/5:\radiusB) {$V_{i+1}$};
\node (3) at (2*\angleA  - \angleA/5:\radiusB)  {$V_{i+2}$};
\node (4) at (3*\angleA  + \angleA/5:\radiusB) {$V_{i+3}$};
\node (5) at (4*\angleA  - \angleA/5:\radiusB) {$V_{i+4}$};
\node (6) at (4*\angleA  + \angleA/5:\radiusB) {$P_{i+1}$};
\node (7)[gray] at (2*\angleA  + \angleA/5:\radiusB)  {$N_{i+3}$};
\node (8)[gray] at (3*\angleA  - \angleA/5:\radiusB) {$N_{i+5}$};

\draw  (B1) -- (1) (B1) -- (2) (B2) -- (3) (B3)--(4) (B4)--(5) (B4) -- (6);

\draw[gray] (7) -- (B2) (8) -- (B3) (0) -- (B0);

\end{tikzpicture}

\end{center}
\caption{Indices of $V_i$, $E_i$, and $N_i$ are taken mod 6, and indices of $P_i$ mod 3. $A - \cdot - B$ indicates that cells $A$ and $B$ share a codimension 1 boundary (represented here by the dot). Dashed gray lines represent boundaries which involve at most one cell associated to an admissible Wilson loop diagram. The solid grey lines correspond to boundaries shared with one of the six $6$-dimensional cells of $\Gr(2,6)$ that are not associated to Wilson loop diagrams. } \label{EPVboundaries}
\end{figure}

Instead of attempting to represent all of these codimension 1 boundaries in one diagram, we describe the shared boundaries of each type of diagram individually; see Figure \ref{EPVboundaries}. For visual clarity, we write $W$ in Figure \ref{EPVboundaries} to represent each positroid cell, rather than $\Sigma(W)$. Since $E_{i,L}$ and $E_{i,R}$ correspond to the same positroid cell, we suppress the $L$ or $R$ subscript and simply refer to these cells as $E_i$.

In Section \ref{sec:cancellingpoles} below we show that the spurious singularities of the integrals $\cI(W)(\cZ_*)$ do cancel exactly on the codimension 1 boundaries shared by pairs of cells associated to Wilson loop diagrams. Figure \ref{EPVboundaries} highlights two types of boundary which will need extra care:

\begin{enumerate}
\item The behavior highlighted in Example \ref{eg:bounds3} above: a 5-dimensional positroid cell in $\Gr(2,6)$ that lies on the boundary of $\Sigma(V_i)$ and is shared with no other Wilson loop diagram. This is represented by the dashed grey lines in Figure \ref{EPVboundaries}.
\item The cell labelled $*$ in Figure \ref{EPVboundaries} lies on the boundary of only one cell associated to a Wilson loop diagram, namely the cell $\Sigma(E_{i,\bullet})$, but this cell is associated to two different Wilson loop diagrams: $E_{i,R}$ and $E_{i,L}$.
\end{enumerate}

\subsection{A graphical device for understanding codimension one boundaries}\label{sub:mnemonic}

In this section we describe how some of the boundaries in Figure \ref{EPVboundaries} can be seen directly from the Wilson loop diagrams.

\begin{dfn}\label{def:boundaryprops}
Let $W = (\cP, n)$ be an admissible Wilson loop diagram, and $p \in \cP$ one of its propagators. For $v \in V_p$, the {\bf boundary propagator} $\D_v p$ is obtained by moving the endpoint of $p$ away from vertex $v$ while maintaining the requirement that no two propagators cross each other, until one of the following occurs:
\begin{enumerate}
\item the endpoint of $p$ reaches another vertex, i.e. $\D_vp$ is supported on $V_p \backslash {v}$; or
\item the endpoint of $p$ touches the endpoint of another propagator $q$.
\end{enumerate}

Define the {\bf boundary diagram} $\D_{p,v}(W)$ to be the diagram obtained from $W$ by replacing propagator $p$ with $\D_vp$. We say that $\D_{p,v}(W)$ is {\bf degenerate} if there is a subset $Q \subset (\cP \setminus p)\cup \D_vp$ such that $|V_Q| < |Q| + 3$.
\end{dfn}

\begin{eg}\label{boundaryeg}
Consider the Wilson loop diagram $V_1$, i.e.
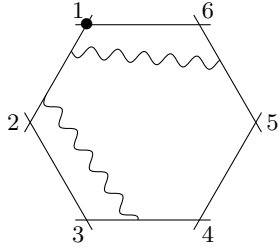
\[\begin{tikzpicture}[rotate=60,baseline=(current bounding box.east)]
	\begin{scope}
	\drawWLD{6}{1.5}
	\drawnumbers
\drawprop{1}{-1}{5}{0}
\drawprop{1}{1}{3}{0}
		\end{scope}
	\end{tikzpicture}\]
Consider the propagator $p = (1,5)$. By replacing $p$ with $\D_2p$ and $\D_1p$ respectively, we obtain examples of both types of boundary diagrams:
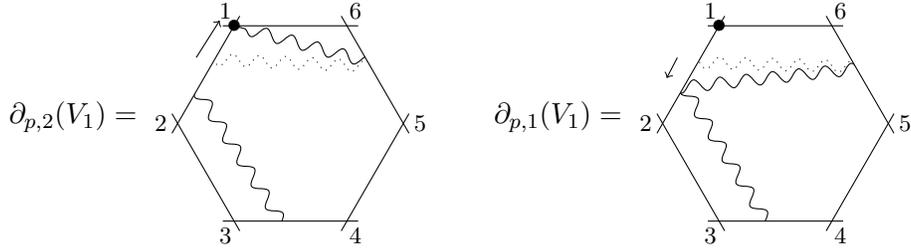
\begin{gather*}\D_{p,2}(V_1)  =
\begin{tikzpicture}[rotate=60,baseline=(current bounding box.east)]
	\begin{scope}
	\drawWLD{6}{1.5}
	\drawnumbers
	\modifiedprop{1}{-0.5}{5}{0}{propagator,dotted}
	\boundaryprop{5}{0}{1}{propagator}
	\drawprop{1}{0.5}{3}{0}
	\boundA{1}{-0.5}{1}
		\end{scope}
	\end{tikzpicture}
\qquad
\D_{p,1}(V_1) =
\begin{tikzpicture}[rotate=60,baseline=(current bounding box.east)]
	\begin{scope}
	\drawWLD{6}{1.5}
	\drawnumbers
	\drawprop{1}{0.3}{3}{0}
	\modifiedprop{5}{0}{1}{-0.3}{propagator,dotted}
	\modifiedprop{5}{0}{1}{1.2}{propagator}
	\boundB{1}{-0.5}{1}
		\end{scope}
	\end{tikzpicture}
\end{gather*}
Let $q = (1,3)$ be the other propagator in $V_1$. Clearly $\D_{p,1}(V_1)$ and $\D_{q,2}(V_1)$ are (combinatorially) the same diagram; we will not distinguish between them. By considering all possible boundary propagators for $V_1$, we see that it has 7 distinct boundary diagrams.
\end{eg}

\begin{eg}\label{eg:Degeneratebound}
Consider the diagram $E_{6,R}$, which has propagators $p = (1,5)$ and $q = (1,4)$. The boundary $\D_{p,5}(E_{6,R})$ is degenerate, since the set $\{q, \D_{5}p\}$ is supported on only $4$ vertices.
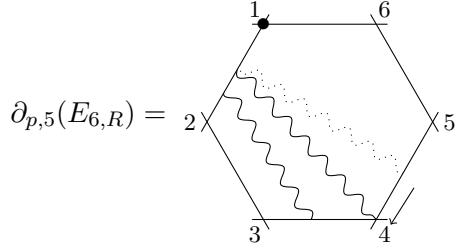
\begin{gather*}\D_{p,5}(E_{6,R})  =
\begin{tikzpicture}[rotate=60,baseline=(current bounding box.east)]
	\begin{scope}
	\drawWLD{6}{1.5}
	\drawnumbers
	\modifiedprop{1}{-0.5}{4}{0}{propagator,dotted}
	\boundaryprop{1}{-0.5}{4}{propagator}
	\drawprop{1}{0.5}{3}{0}
	\boundA{4}{-0.5}{4}
		\end{scope}
	\end{tikzpicture}
\end{gather*}
The diagram $\D_{q,3}(E_{6,R})$ is also degenerate, leaving $E_{6,R}$ with 5 distinct nondegenerate boundary diagrams.
\end{eg}

Since the support of a propagator determines which entries of $\cC(W)$ are nonzero, we can give an intuitive interpretation of Definition \ref{def:boundaryprops} in terms of $\cC(W)$:

\begin{dfn}
Define $\cC(\D_{p,v}(W))$ to be the matrix obtained by applying Definition \ref{Cstardef} to the diagram $\D_{p,v}(W)$. In other words, if $p$ is no longer supported on $v$ in $\D_{p,v}(W)$, then $c_{p,v} = 0 $ in $C(\D_{p,v}(W))$, while if $p$ and $q$ meet in $\D_{p,v}(W)$ (both lying between vertices $v$ and $v+1$, say) then $c_{q,v}$ and $c_{q,v+1}$ are constrained by the condition that $c_{p,v}c_{q,v+1} - c_{q,v}c_{p,v+1} = 0$ in $\cC(\D_{p,v}(W))$.
\end{dfn}

We write $\Delta_{p,v}(W)$ for the minor of $\cC(W)$ that is set to 0 in $\cC(\D_{p,v}(W))$; that is,
\[\Delta_{p,v}(W) =\left\{ \begin{array}{ll}
c_{p,v} & \text{ if $p$ is no longer supported on $v$ in $\D_{p,v}(W)$}; \\
c_{p,v}c_{q,v+1} - c_{q,v}c_{p,v+1} & \text{ if propagators $p$ and $q$ touch in $\D_{p,v}(W)$.}
\end{array}\right.\]
Using this notation, we can write \bas \cC(\D_{p,v}(W)) = \lim_{\Delta_{p,v} \rightarrow 0} \cC(W) \;. \eas We call $\cC(\D_{p,v}(W))$ a {\bf boundary matrix} of $\cC(W)$.

\begin{rmk}\label{rmk:bdnyspoles} By construction, these $\Delta_{p,v}(W)$ are exactly the factors of $R(W)$. In the case of Wilson loop diagrams with $2$ propagators and $6$ vertices, we may exhaustively check that the factors of $R(W)$ that correspond to degenerate boundaries are exactly those that correspond to non-simple poles of $I(W)(\cZ_*)$. 
\end{rmk}

We are now ready to introduce a graphical device for calculating boundaries of Wilson loop diagrams.

\begin{mnemonic}\label{mnemonic}
Let $W = (\cP, n) $ and $W' = (\cP', n)$ be two Wilson loop diagrams. If there exist two vertex propagators pairs $(p,v)$ and $(p', v')$, with $p \in \cP$, $v \in V_p$ and $p' \in \cP'$, $v' \in V_{p'}$ such that \bas \cC(\D_{p,v} (W))= \cC(\D_{p',v'} (W'))\;,\eas then the corresponding cells $\Sigma(W)$ and $\Sigma(W')$ share a codimension 1 boundary in $\Gr(k,n)$.
\end{mnemonic}

This is a slightly weaker condition than requiring the boundary diagrams themselves to be equal. Certainly it can happen that two boundary diagrams are equal; for instance, recall the boundary shared between $V_1$ and $E_{6,R}$ from Example \ref{eg:bounds1}. In this case, is easy to see that \[\D_{(1,5), 6}(V_1) = \D_{(1,4), 4}(E_{6,R}),\] and hence the corresponding matrices $\cC(\D_{(1,5), 6}(V_1))$ and $\cC(\D_{(1,4), 4}(E_{6,R}))$ are equal as well.

On the other hand, recall the ``obvious'' boundary between $V_1$ and $P_1$ in Example \ref{eg:bounds2}: while we do have an equality $\cC(\D_{(1,5), 2}(V_1))=\cC(\D_{(4,6), 4}(P_1))$ at the level of the matrices, the boundary diagrams $\D_{(1,5), 2}(V_1)$ and $\D_{(4,6), 4}(P_1)$ are not equal: \begin{equation}\label{eq:slidingpropagator}\D_{(1,5),2}(V_1)  =
\begin{tikzpicture}[rotate=60,baseline=(current bounding box.east)]
	\begin{scope}
	\drawWLD{6}{1.5}
	\drawnumbers
	\modifiedprop{1}{-0.5}{5}{0}{propagator,dotted}
	\boundaryprop{5}{0}{1}{propagator}
	\drawprop{1}{0.5}{3}{0}
		\end{scope}
	\end{tikzpicture}
\qquad
\D_{(4,6),4}(P_1) =
\begin{tikzpicture}[rotate=60,baseline=(current bounding box.east)]
	\begin{scope}
	\drawWLD{6}{1.5}
	\drawnumbers
	\drawprop{1}{0}{3}{0}
	\modifiedprop{6}{0}{4}{0}{propagator,dotted}
	\boundaryprop{6}{0}{5}{propagator}
		\end{scope}
	\end{tikzpicture}
\end{equation}
However, note that the two boundary diagrams only differ in one propagator, and we can obtain one from the other by ``sliding'' the boundary propagator along half the length of an edge.

\begin{rmk} This ``sliding'' of the boundary propagator yielding an equivalent diagram is a general phenomenon: if we have two Wilson loop diagrams $W = (Q\cup (i, i+2) , n)$ and $W' = (Q\cup (i-1, i+1) , n)$, then the boundary diagrams $\D_{(i,i+2), i}(W)$ and $\D_{(i-1,i+1), i+1}(W')$ differ by a half-edge slide and the boundary matrices are equal.\end{rmk}

Graphical Prompt \ref{mnemonic} was originally proposed as a method of identifying all shared boundaries between pairs of admissible Wilson loop diagrams. However, direct computation yields examples of shared boundaries which are {\em not} seen by this graphical approach. We give two examples to illustrate this phenomenon: Example \ref{eg:Vibdnys}, which was already known, and Example \ref{eg:mnemonicVEbound}, which was only identified by the authors when they started to systematically apply the tools of total positivity to this question.

\begin{eg} \label{eg:Vibdnys} Let $p$ be the propagator $(1,5)$. This is a propagator present in both $V_1$ and $V_5$. Consider the boundary diagrams
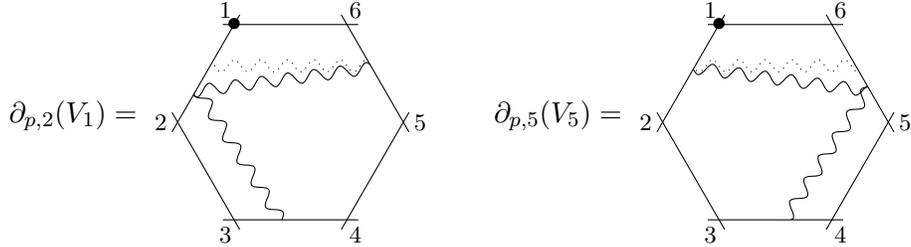
\begin{gather*}\D_{p,2}(V_1)  =
\begin{tikzpicture}[rotate=60,baseline=(current bounding box.east)]
	\begin{scope}
	\drawWLD{6}{1.5}
	\drawnumbers
	\modifiedprop{5}{0}{1}{0}{propagator,dotted}
	\drawprop{5}{0}{1}{1.5}
    \drawprop{1}{0.5}{3}{0}	
		\end{scope}
	\end{tikzpicture}
\qquad
\D_{p,5}(V_5) =
\begin{tikzpicture}[rotate=60,baseline=(current bounding box.east)]
	\begin{scope}
	\drawWLD{6}{1.5}
	\drawnumbers
	\modifiedprop{1}{0}{5}{0}{propagator,dotted}
	\drawprop{1}{0}{5}{-1.5}
    \drawprop{5}{-0.5}{3}{0}	
		\end{scope}
	\end{tikzpicture}
\end{gather*} The corresponding matrices are \bmls C(\D_{p,2}(V_1)) = \left(
\begin{array}{cccccc}
\lambda c_{q,1} & \lambda c_{q,2}  & 0 & 0 & c_{p,5} & c_{p,6} \\
c_{q,1} & c_{q,2} & c_{q,3} & c_{q,4} & 0 & 0 \\
\end{array}
\right)  \quad \textrm{and} \\ C(\D_{p,6}(V_5)) = \left(
\begin{array}{cccccc}
c_{p,1} & c_{p,2}  & 0 & 0 & \mu c_{r,5} & \mu c_{r,6} \\
0 & 0 & c_{r,3} & c_{r,4} & c_{r,5} & c_{r,6} \\
\end{array}
\right), \emls
where $\lambda, \mu \in \R^{\times}$. Even though these two boundary matrices are not equal, it is easily verified that they have the same sets of independent column vectors. Thus they define the same positroid, and hence the same $5$-dimensional cell of $\Gr(2,6)$, corresponding to the Le diagram \[\ytableaushort{0+0+,+++}.\]
\end{eg}

\begin{eg}\label{eg:mnemonicVEbound}
The positroid cells $\Sigma(V_1)$ and $\Sigma(E_{5,\bullet})$ correspond to the Le diagrams
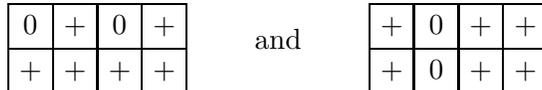
\[
\ytableaushort{0+0+,++++} \qquad \text{ and }\qquad  \ytableaushort{+0++,+0++} \]
respectively. Keeping Example \ref{eg:badLebehavior} in mind, we see that they share a 5-dimensional boundary, namely
 \begin{equation}\label{eq:5dimcelleg}\ytableaushort{+00+,+0++}.\end{equation}
This can be realized as the cell associated to a boundary diagram for each of the diagrams $V_1$, $E_{5,L}$, and $E_{5,R}$, but in a way that is completely missed by the graphical representation. The three boundary diagrams which yield the Le diagram in \eqref{eq:5dimcelleg} are:
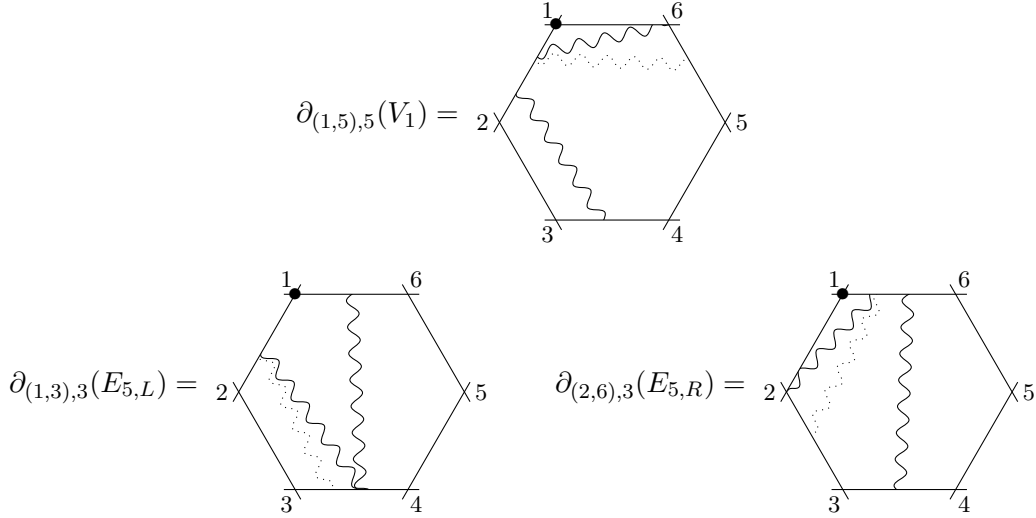
\begin{gather*}
\D_{(1,5),5}(V_1) =
\begin{tikzpicture}[rotate=60,baseline=(current bounding box.east)]
	\begin{scope}
	\drawWLD{6}{1.5}
	\drawnumbers
	\modifiedprop{1}{-0.5}{5}{0}{propagator,dotted}
	\boundaryprop{1}{-0.5}{6}{propagator}
	\drawprop{1}{0.5}{3}{0}
		\end{scope}
	\end{tikzpicture}
\\
\D_{(1,3),3}(E_{5,L}) = \begin{tikzpicture}[rotate=60,baseline=(current bounding box.east)]
	\begin{scope}
	\drawWLD{6}{1.5}
	\drawnumbers
	\modifiedprop{1}{0}{3}{-0.8}{propagator,dotted}
	\modifiedprop{1}{0}{3}{1}{propagator}
	\drawprop{6}{0}{3}{0.5}
		\end{scope}
	\end{tikzpicture}
	\qquad
\D_{(2,6),3}(E_{5,R}) =
\begin{tikzpicture}[rotate=60,baseline=(current bounding box.east)]
	\begin{scope}
	\drawWLD{6}{1.5}
	\drawnumbers
	\modifiedprop{6}{0.5}{2}{0}{propagator,dotted}
	\boundaryprop{6}{0.5}{2}{propagator}
	\drawprop{6}{-0.5}{3}{0}
		\end{scope}
	\end{tikzpicture}
\end{gather*}
Unlike the example in equation \eqref{eq:slidingpropagator}, there is no obvious relationship between these three diagrams.
\end{eg}

For physicists, only the study of the boundaries obtained via boundary propagators is of interest: they encode the spurious singularities of the integral $\cI(W)$. It is therefore important to have a way of identifying all such boundaries.

For the case $k=2$, $n=6$, the following result completely characterizes the shared boundaries obtained from propagator moves.

\begin{prop}\label{res:bounds}
$B$ is a $5$-dimensional cell in $\cW(2,6)$ satisfying $B \subseteq \Sigma(W) \cap \Sigma(W')$ for two distinct Wilson loop diagrams $W$ and $W'$ if and only if \begin{enumerate}
\item The cell $B$ can be realized as the cell parametrized by some boundary diagram $\D_{\hat{p}, \hat{v}}(\hat{W})$, with $\hat{W}$ an admissible Wilson loop diagram with $2$ propagators on $6$ vertices.
\item The minor $\D_{\hat{v}, \hat{p} }(\hat{W})$ corresponds to a simple pole of $I(\hat{W})(\cZ_*)$.
\end{enumerate}
\end{prop}
\begin{proof}
This is verified by direct calculation, by computing all the $5$-dimensional cells contained in $\cW(2,6)$ (as shown in Figure \ref{EPVboundaries}) and all the boundary diagrams for admissible Wilson loop diagrams with $2$ propagators on $6$ points.
\end{proof}

In light of Proposition \ref{res:bounds}, and preliminary computations for higher $k$ and $n$, we make the following conjecture.

\begin{conj}$B$ is a $(3k-1)$-dimensional cell in $\cW(k,n)$ satisfying $B \subseteq \Sigma(W) \cap \Sigma(W')$ for two distinct Wilson loop diagrams $W$ and $W'$ if and only if \begin{enumerate}
\item The cell $B$ can be realized as the cell parametrized by some boundary diagram $\D_{\hat{p}, \hat{v}}(\hat{W})$, with $\hat{W}$ an admissible Wilson loop diagram with $k$ propagators on $n$ vertices.
\item The minor $\D_{\hat{v}, \hat{p} }(\hat{W})$ corresponds to a simple pole of $I(\hat{W})(\cZ_*)$.
\end{enumerate}\end{conj}

\begin{rmk} $\cW(2,6)$ is not simply the 6-skeleton of $\Gr(2,6)$. Indeed, each of the cells $\Sigma(V_i)$ admits a boundary in $\Gr(2,6)$ which cannot be realized as a the cell of a boundary diagram. These are the dashed gray boundaries seen in Figure \ref{EPVboundaries}, which are shared in $\Gr(2,6)$ only with cells that do not correspond to admissible Wilson loop diagrams.\end{rmk}

\begin{rmk} While the diagrams $E_{i,L}$ and $E_{i,R}$ both give rise to the same cell, Proposition \ref{res:bounds} indicates that they share a boundary. This is the boundary labelled $*$ in Figure \ref{EPVboundaries}, and it will play an important role in Section \ref{sec:cancellingpoles} when we consider the cancellation of spurious poles on codimension 1 boundaries. However, this 5-dimensional cell lies on the boundary of the space $\cW(2,6)$, as it is shared by the cells $\Sigma(E_{i,\bullet})$ and $\Sigma(N_{i+4})$ only.\end{rmk}

We conclude this section by providing a different intepretation of Example \ref{eg:mnemonicVEbound} in terms of propagator moves. Indeed, to be able to fully graphically predict the geometric relationship between this boundary cell and the Wilson loop diagrams it borders, one needs to step briefly into the world of inadmissible diagrams. We provide this as an example of the complexity present in the geometry even in this simple case, not to advocate including non-admissible diagrams into the theory.

\begin{eg}
Consider the non-admissible Wilson loop diagram
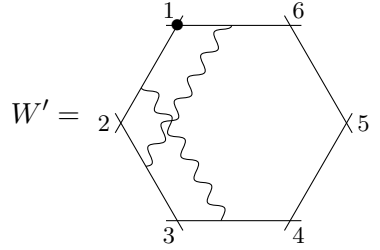
\[
W' = \begin{tikzpicture}[rotate=60,baseline=(current bounding box.east)]
	\begin{scope}
	\drawWLD{6}{1.5}
	\drawnumbers
	\drawprop{1}{0}{3}{0}
	\drawprop{2}{0}{6}{0}

		\end{scope}
	\end{tikzpicture} .
\]
Computing the positroid associated with $W'$ yields the cell $\Sigma(E_5)$. Direct computation shows that
\[\cC(\D_{(1,5), 5}(V_1))= \cC(\D_{(2,6),3}(W')),\]
i.e. according to Graphical Prompt \ref{mnemonic} we would expect $V_1$ and $W'$ to share a boundary. Indeed, if we draw the two boundary diagrams we obtain
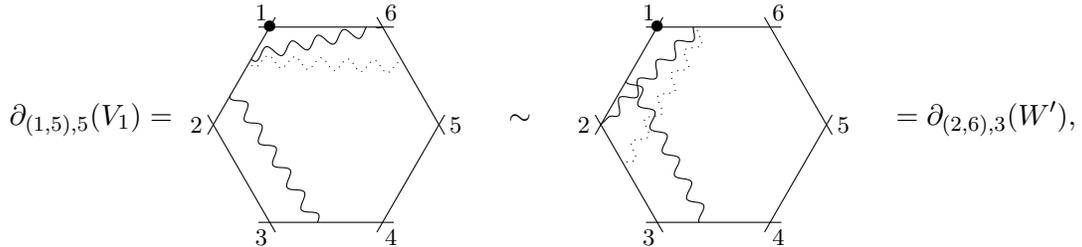
\[
\D_{(1,5),5}(V_1) = \begin{tikzpicture}[rotate=60,baseline=(current bounding box.east)]
	\begin{scope}
	\drawWLD{6}{1.5}
	\drawnumbers
	\modifiedprop{1}{-0.5}{5}{0}{propagator,dotted}
	\boundaryprop{1}{-0.5}{6}{propagator}
	\drawprop{1}{0.5}{3}{0}
		\end{scope}
	\end{tikzpicture}
\quad
\sim
\quad
\begin{tikzpicture}[rotate=60,baseline=(current bounding box.east)]
	\begin{scope}
	\drawWLD{6}{1.5}
	\drawnumbers
	\modifiedprop{6}{0}{2}{0}{propagator,dotted}
	\boundaryprop{6}{0}{2}{propagator}
	\drawprop{1}{-0.5}{3}{0}

		\end{scope}
	\end{tikzpicture}
	\quad
= \D_{(2,6),3}(W'),
\]
i.e. the two diagrams are related by the same ``half-edge slide'' propagator move seen in equation \eqref{eq:slidingpropagator}.

Diagrams with crossing propagators are meaningless from a physical point of view, but this suggests they may be a useful tool to study the combinatorics and geometry of Wilson loop diagrams in the future.
\end{eg}

\subsection{The homology of $\cW(2,6)$}\label{sub:2,6homology}

Figure \ref{EPVboundaries} gives an insight into the geometry of the space $\cW(2,6)$. In this section, we discuss the geometry in more detail and compute the homology of this space. Note that this is \emph{not} the cohomology of the Amplituhedron. Nor is it the homology of the larger space $\cW_*(2,6)$ that is conjectured to be the preimage of the Amplituhedron. We only consider the positive subspace $\cW(2,6)$ tiled by the positroid cells.

Recall from Section \ref{sec:Feynmanrules} that the Wilson loop diagrams define a subspace $\cW_*(2,6) \subseteq \mathbb{G}_{\R}(2, 7)$ parametrized by the matrices $\cC_*(W)$. The external data $\cZ_* \in M(7, 6)$ defines a projection \bas \cZ_*: \G_{\R}(2, 7) \rightarrow \Gr(2, 6)\;,\eas that can be restricted onto the subspace $\cW_*(2,6)$. The holomorphic Wilson loop \[\mathfrak{W}_{2,6} = \sum_{\substack{W \; \textrm{admis.,} \\ \textrm{2 props., 6 verts.}}} \cI(W)(\cZ_*)\] assigns a volume to the projection $\cZ_*(\cW_*(2,6))$.

Recall from Definition \ref{dfn:WLDcomplex} that \bas \cW(2,6) = \bigcup_{\substack{W \; \textrm{admis.,} \\ \textrm{2 props., 6 verts.}}} \overline{\Sigma(W)} \;.\eas The Le diagrams associated to these cells are listed in Table \ref{nameWLDLetable}. By counting the $+$ symbols in each Le diagram, we see that $\cW(2,6)$ is a 6-dimensional submanifold of $\Gr(2,6)$. The following facts about $\cW(2,6)$ follow by direct computation.

\begin{enumerate}
\item There are six $6$-dimensional cells in $\Gr(2,6)$ that are not part of $\cW(2,6)$. These are denoted $N_i$ in Figure \ref{missingcellsfigure}.
\item There are six $5$-dimensional cells in $\Gr(2,6)$ lying on the boundaries between $N_i$ and $N_{i+1}$ which are also not in $\cW(2,6)$.
\item All other cells of dimension $\leq 5$ in $\Gr(2,6)$ are contained in $\cW(2,6)$.
\item The manifold $\cW(2,6)$ is not closed.  The boundary of the manifold consists of exactly the twelve $5$-dimensional cells indicated in Figure \ref{EPVboundaries} that are shared by exactly one cell defined by a Wilson loop diagram and a cell $N_i$.
\end{enumerate}

While there are several conjectures about the homology of the Amplituhedron  \cite{UnwindingAmplituhedron}, there is little know about the preimage of this space before the projection imposed by the external particle data, either in the BCFW or Wilson loop context. With the data described above in hand, we are able to compute the homology of $\cW(2,6)$ directly.

\begin{thm}\label{res:homology}
The homology groups of $\cW(2,6)$ are as follows: \bas H_i(\cW(2,6)) =\begin{cases} \R & \textrm{if } i \in 0, 5 \\ 0 & \textrm{else}.  \end{cases} \eas
\end{thm}
\begin{proof}
This result was obtained by direct computation, using Python and Sympy \cite{sympy} to obtain the basis sets for each positroid and to compare each pair of basis sets, and the Chain Complexes module of the computer algebra system Sage \cite{sagemath} to compute the homology. The code used by the authors is provided as an auxiliary file on arXiv.
\end{proof}

\section{The cancellation of spurious poles \label{sec:cancellingpoles} }

In this section, we return to the question posed in Section \ref{sec:Feynmanrules}: verifying that the spurious singularites of the Wilson loop diagrams with $2$ propagators on $6$ points cancel in the holomorphic Wilson loop calculation given in \eqref{holowilsonloop}.

Recall from Proposition \ref{res:bounds} that in $\cW(2,6)$, \begin{enumerate}
\item The non-degenerate boundaries $\D_{p,v}(W)$ correspond to degree one factors of $R(W)(\cZ_*)$, i.e. spurious poles of $\cI(W)(\cZ_*)$. These poles are denoted $\Delta_{p,v}(W)$.
\item The limits $\lim_{\Delta_{p,v}(W) \rightarrow 0} C(W)$ correspond exactly to the codimension $1$ boundaries of $\Sigma(W)$ that are shared with some cell $\Sigma(W')$ with $W' \neq W$. (Recall that while the Wilson loop diagrams $W$ and $W'$ are different, their corresponding positroid cells may be the same.)
\end{enumerate}

In this section, we prove the following theorem.

\begin{thm} \label{boundarycancelthm}
Let $W$ be an admissible Wilson loop diagram in with $2$ propagators and $6$ external particles. Let $\Sigma'$ be any $5$-dimensional boundary of the $6$-dimensional cell $\Sigma(W)$. Then
\ba \sum_{\begin{subarray}{c} \D_{p,v}(W') = \Sigma', \\ W' \textrm{ admiss. } \end{subarray}} \Res_{\Delta_{p,v}(W') \rightarrow 0}\cI(W')(\cZ_*) =0\label{boundarycanceleq}\;.\ea\end{thm}

In other words, we show that the residues of the integrals $\Res_{\Delta_{p,v}(W) \rightarrow 0}\cI(W)(\cZ_*)$ cancel exactly on the $5$-dimensional cells $\lim_{\Delta_{p,v}(W) \rightarrow 0} C(W)$. A direct corollary of this is that the sum of the spurious poles cancel in the sum
\[\mathfrak{W}_{2,6} = \sum_{\substack{W \; admiss.\\ n\; verts, k\; props}} \cI(W)(\cZ_*)\]

We begin with an outline of the proof, which is proved in three cases.

Case 1 considers the types of propagators moves described in Section \ref{sub:mnemonic}. In this simplest case, we consider two Wilson loop Diagrams $W = (\{p,q\}, 6)$ and $W' = (\{p',q\}, 6)$, where $p$ and $p'$ are such that $\cC(\D_{p,v}(W)) = \cC(\D_{p',v'}(W'))$. For a picture of the type of computations considered in this case, see Example \ref{eg:bounds2}. This case captures the boundary between $P_i$ and $V_i$, $V_{i+2}$, $V_{i+3}$, or $V_{i+5}$; the boundaries between $P_i$ and $E_{i+1, L}$, $E_{i+2, R}$, $E_{i+4, L}$, or $E_{i+5, R}$; the boundaries between $V_i$ and $E_{i-1, R}$ or $E_{i+2, L}$; and the boundaries between $E_{i, R}$ and $E_{i,L}$ indicated in Figure \ref{EPVboundaries}. The fact that $\Sigma(E_{i,R}) = \Sigma(E_{i,L})$ then explains why certain triples of the form $\Sigma(P_i)$, $\Sigma(E_j)$ and $\Sigma(V_k)$ share a common boundary.

Case 2 considers what happens when bringing together two propagators on a boundary edge. In this case, one is not simply setting a single parameter to $0$, but instead setting a $2 \times 2$ minor of $\cC(W)$ to $0$. The actual calculation for this case requires a change of basis in the integral. For a picture of the type of computations considered in this case, see Example \ref{eg:Vibdnys}. This accounts for the three way boundary between $V_i$, $V_{i+2}$ and $V_{i+4}$ indicated in Figure \ref{EPVboundaries}.

Case 3 handles the type of boundary highlighted in Example \ref{eg:mnemonicVEbound}. As shown in that example, there are two spurious poles of $\cI(V_i)$, namely $\Delta_{(i,i-2), i-2}(V_i)$ and $\Delta_{(i,i+2), i+2}(V_i)$, that correspond to a $5$-dimensional boundary of the cells $\Sigma(E_{i-2})$ and $\Sigma(E_{i-3})$ respectively. However, the fact that this boundary sits between these two cells is missed by Graphical Prompt \ref{mnemonic}, indicating that something more complicated is happening here. For a picture of the type of computations considered in this case, see Example \ref{eg:mnemonicVEbound}.

\begin{proof}[Proof of Theorem \ref{boundarycancelthm}]
Recall from \eqref{localizedint} that for $W = (\{p,q\}, 6)$, the integral $\cI(W)$ localized at $\cZ_*$ is given by \bas \cI(W)(\cZ_*) = \frac{F_p^1F_q^2}{R(W)(\cZ_*)} \;,\eas where the $F_p^b$ are defined in $\eqref{Feq}$, and $R(W)(\cZ_*)$ is computed directly after Definition \ref{def:RW}. The proof proceeds by considering the three cases outlined above. Inspection of Figure \ref{EPVboundaries} shows that these cases are exhaustive.

\textbf{Case 1}: In this case, there exist a pair of Wilson loop diagrams $W = (\{p,q\}, 6)$ and $W' = (\{p',q\}, 6)$ such that $\cC(\D_{p,v}(W)) = \cC(\D_{p',v'}(W'))$. After localization, we have $\Delta_{v,p}(W) = \sigma_{p,v}$ and $\Delta_{v',p'}(W') = \sigma_{p',v'}$. Notice that $|V_p \cap V_{p'}| = 3$, and that $\{v\} = V_p \setminus V_{p'}$ and $\{v'\} = V_{p'} \setminus V_{p}$. In this case, it follows from the definition of $\sigma_{b,a}$  that
\bas \sigma_{p,v} = -\sigma_{p',v'}\;. \eas
Also, in the limit $\sigma_{p,v}\rightarrow 0$, we may write $Z_*^\mu$ as a linear combination of the $Z_w^\mu$ with $w \in V_p$. That is,
\ba Z_*^\mu = \sum_{w \in V_p; w \neq v} \alpha_w Z_w^\mu \quad \alpha_w \in \R \:. \label{usefultrick}\ea
In this limit, for any $w \in V_p \cap V_{p'}$, it follows from \eqref{usefultrick} that $\sigma_{p, w} = \alpha_w\langle i_p (i_p+1) j_p (j_p+1) \rangle$ and $\sigma_{p', w}=-\alpha_w\langle i_{p'}(i_{p'}+1) j_{p'}(j_{p'}+1) \rangle$. We can use this to write
\bas \lim_{\sigma_{p,v}\rightarrow 0}F_p^1 = \left(\sum_{w \in V_p \cap V_{p'}} \alpha_w Z_w^\mu\right)^4\langle i_p (i_p+1) j_p (j_p+1) \rangle^4 \;.\eas
Similarly,
\bas \lim_{\sigma_{p',v'}\rightarrow 0}F_{p'}^1 = \left(\sum_{w \in V_p \cap V_{p'}} -\alpha_w Z_w^\mu\right)^4\langle i_{p'} (i_{p'}+1) j_{p'} (j_{p'}+1) \rangle^4 \;.\eas
Therefore, after making the above substitutions at the appropriate limits, one sees exactly that \bas \lim_{\sigma_{p,v} \rightarrow 0} \big(\cI(W)(\cZ_*)\big) = -\lim_{\sigma_{p',v'} \rightarrow 0} \big(\cI(W')(\cZ_*)\big) .\eas

\textbf{Case 2}: In this case, and the next, one needs to first perform a careful change of variables before localization. Then one may proceed by calculations in the same vein as in Case 1. In this case, we consider what happens when bringing two propagators to meet on an edge, as in the boundary diagram $\D_{(i,i+4), i}(V_i)$ for example. We begin by rewriting the matrices $\cC_*(V_i)$, $\cC_*(V_{i+2})$ and $\cC_*(V_{i-2})$ such that the minors $\Delta_{i, (i, i+4)}(V_i)$, $\Delta_{i+2, (i+2, i)}(V_{i+2})$ and $\Delta_{i-2, (i-2, i+2)}(V_{i-2})$ can each be expressed by a single variable. In other words, we chose $\alpha, \alpha_+, \alpha_-, f, f_+$ and $f_-$ to be real variables such that
\bmls \cC_*(V_i) = \begin{bmatrix} 1& \cdots & c_{1,i} & c_{1,i+1} & \cdots \\1&  \cdots & \alpha c_{1,i} & \alpha c_{1,i+1} +f_+ & \cdots
\end{bmatrix} \quad ; \quad \cC_*(V_{i+2}) = \begin{bmatrix} 1& \cdots & c_{1,i+2} & c_{1,i+3}  & \cdots \\1&  \cdots & \alpha_+ c_{1,i+2} & \alpha_+ c_{1,i+3} +f_- & \cdots  \end{bmatrix} \\ \textrm{ and } \cC_*(V_{i-2}) = \begin{bmatrix} 1& \cdots & c_{1,i-2} & c_{1,i-1} & \cdots \\1&  \cdots & \alpha_- c_{1,i-2} & \alpha_- c_{1,i-1}+f & \cdots
\end{bmatrix} \;. \emls
Under this notation, the common $5$-dimensional cell, $\Sigma(V_i) \cap \Sigma(V_{i+2}) \cap \Sigma(V_{i-2})$, can be represented by the limit \bas \lim_{f \rightarrow 0} \cC_*(V_i) = \lim_{f_+ \rightarrow 0} \cC_*(V_{i+2}) = \lim_{f_- \rightarrow 0} \cC_*(V_{i-2}) \;. \eas  Rewriting the integrals $\cI(V_i)(\cZ_*)$ under this new parametrization, we obtain \bas \cI(V_i) = \int \frac{ \rd c_{1,i-2} \rd c_{1,i-1} \rd c_{1,i} \rd c_{1,i+1} \rd \alpha \rd f \rd  c_{2,i+2} \rd c_{2,i+3}}{ c_{1,i-2} c_{1,i-1} c_{1,i}  c_{1,i+1}  \alpha  f  c_{2,i+2} c_{2,i+3}} \delta^{8|8}(\cC_*(V_i) \cdot \cZ_*) \;,\eas \bas \cI(V_{i+2})(\cZ_*) = \int \frac{ \rd c_{1,i} \rd c_{1,i+1} \rd c_{1,i+2} \rd c_{1,i+3} \rd \alpha_+ \rd f_+ \rd  c_{2,i+4} \rd c_{2,i+5}}{ c_{1,i} c_{1,i+1} c_{1,i+2}  c_{1,i+3}  \alpha_+  f_+  c_{2,i+4} c_{2,i+5}} \delta^{8|8}(\cC_*(V_{i+2}) \cdot \cZ_*) \;,\eas and \bas \cI(V_{i-2})(\cZ_*) = \int \frac{ \rd c_{1,i-2} \rd c_{1,i-1} \rd c_{1,i} \rd c_{1,i+1} \rd  c_{2,i-4} \rd c_{2,i-3} \rd \alpha_- \rd f_- }{ c_{1,i-2} c_{1,i-1} c_{1,i}  c_{1,i+1}  c_{2,i-2} c_{2,i-3} \alpha_-  f_-  } \delta^{8|8}(\cC_*(V_{i-2}) \cdot \cZ_*) \;.\eas

In this calculation, since the respective boundaries send the $2\times 2$ minors $\Delta_{i, (i, i+4)}(V_i)$, $\Delta_{i+2, (i+2, i)}(V_{i+2})$, and $\Delta_{i-2, (i-2, i+2)}(V_{i-2})$ to $0$, one cannot employ the computational trick of \eqref{usefultrick} to compare $F_p^1$ to $F_{p'}^1$. Therefore, we need to introduce a change of parametrization of $\cC_*(V_{i+2})$ and $\cC_*(V_{i-2})$.
Given an appropriate choice of ordering of the propagators of each diagram,\footnote{Note that this process requires a choice of ordering on the propagators of $V_i$. We use the convention that the propagator $(j, j+2)$ defines the second row of the matrix $C_*(V_j)$ for $j = i$ or $j = i+2$, while the same propagator defines the first row of the associated matrix when $j = i-2$.
} define \ba \cC_*'(V_{i+2}) = \begin{bmatrix}\frac{-\alpha_+}{1-\alpha_+} & \frac{1}{1-\alpha_+} \\ 1 & 0 \end{bmatrix} \cC_*(V_{i+2})\;. \label{+changeofvars} \ea and \ba \cC_*'(V_{i-2}) = \begin{bmatrix}1 & 0  \\ \frac{-\alpha_-}{1-\alpha_-} & \frac{1}{1-\alpha_-}  \end{bmatrix} \cC_*(V_{i-2})\;. \label{-changeofvars} \ea Notice that the matrices $\lim_{f_+\rightarrow 0} \cC_*'(V_{i+2})$ and $\lim_{f_- \rightarrow 0}\cC_*'(V_{i-2})$ have the same form as $\lim_{f\rightarrow 0}\cC_*(V_i)$. That is, all three matrices have $0$s, $1$s and variables in the same entries.

Let $\cI (\cC_*'(V_{i+2}))$ and $\cI (\cC_*'(V_{i-2}))$ be the two integrals $\cI(V_{i+2})$ and $\cI(V_{i-2})$ rewritten in terms of the variables entries of $\cC_*'(V_{i+2})$ and $\cC_*'(V_{i-2})$ respectively.  Evaluating these integrals (as in Case 1, or as demonstrated explicitly in Example \ref{complicateddiagrameg}), we get that \bas \lim_{f_+ \rightarrow 0} \cI (V_{i+2}) \lim_{f_+ \rightarrow 0} \cI (\cC'_*(V_{i+2})) = \frac{-1}{1 - \sigma_{2,i}/\sigma_{1,i}} \lim_{f \rightarrow 0} \cI(V_i) \eas and \bas \lim_{f_- \rightarrow 0} \cI (V_{i-2}) \lim_{f_- \rightarrow 0} \cI (\cC'_*(V_{i-2})) = \frac{\sigma_{2,i}/\sigma_{1,i} }{1 - \sigma_{2,i}/\sigma_{1,i}} \lim_{f \rightarrow 0}\cI(V_i)  \;.\eas Therefore, $\lim_{f\rightarrow 0} \cI(V_i) + \lim_{f_+\rightarrow 0} \cI(V_{i+2}) + \lim_{f_-\rightarrow 0} \cI(V_{i-2}) = 0$, as desired.

\textbf{Case 3}: The final case is very similar to the second case, but with a different change of variables which exploits the fact that $\Sigma(E_{i,R}) = \Sigma(E_{i,L})$. The purpose of this change of variables is to write $\cI(E_{i,R})$ and $\cI(E_{i,L})$ under a common parametrization of this cell, and thus associate a single integral to it. We introduce a new matrix $\cC_*(E_i)$, given by  \bas \cC_*(E_i) = \begin{bmatrix} 1 & \ldots & c_{1,i+1} & c_{1,i+2} & c_{1,i+3}& c_{1,i+4} & 0 &\ldots \\ 1 & \ldots & 0 &c_{2,i+2} & c_{2,i+3} &c_{2,i+4}&c_{2,i+5} & \ldots  \end{bmatrix} \;,  \eas where the $c_{b,a}$ are real variables as usual. Note that this matrix is yet another parametrization of the cell $\Sigma(E_i)$. Write
\bas \cI\big(\cC_*(E_i)\big) = \int_{(\R^4)^2} \frac{\rd c_{1,i+1} \rd c_{1,i+2} \rd c_{1,i+3} \rd c_{1,i+4} \rd c_{2,i+2} \rd c_{2,i+3} \rd c_{2,i+4} \rd c_{2,i+5}}{\begin{multlined}c_{1,i+1} c_{2,i+5} c_{2,i+2} c_{1,i+4}(c_{1,i+2}c_{2,i+3}-c_{2,i+2}c_{1,i+3})\cdot\\(c_{1,i+3}c_{2,i+4}-c_{1,i+4}c_{2,i+3})\end{multlined}}\delta^{8|8}(C_*(E_i)\cdot \cZ_*)\;.\eas
By direct calculations of the form employed in Case 1, we see that \bas \lim_{\sigma_{1, i+4} \rightarrow 0} \cI(\cC_*(E_i))(\cZ_*)  & = - \lim_{c_{(i, i+2), i} \rightarrow 0} \cI(V_{i+2}) \\ \lim_{c_{2, i+2} \rightarrow 0} \cI(\cC_*(E_i))  & = - \lim_{c_{(i+3, i+5), i} \rightarrow 0} \cI(V_{i+3})\;.\eas However, $\cI(\cC_*(E_i))$ is a different integral than either $\cI(E_{i,R})$ or $\cI(E_{i,L})$. The rest of this proof proceeds by showing that \bas \cI(\cC_*(E_i)) =\cI(E_{i,R}) + \cI(E_{i,L}) \;.\eas

We proceed by a change of variables as in Case 2. We rewrite the integrals $\cI(E_{i,R})$ and $\cI(E_{i,L})$ in terms of the parametrization of $\Sigma(E_i)$ given by the matrix $\cC_*(E_i)$. As before, we fix an ordering of the propagators in both $\cC_*(E_{i,R})$ and $\cC_*(E_{i, L})$. Let the first row of both matrices correspond to the propagator $(i+1, i + 4)$. Under this choice of ordering, write \bas \cC_*'(E_{i,R}) = \begin{bmatrix}0 & 1 \\ \frac{c_{2,i+1}}{c_{2,i+1} - c_{1,i+1}} & \frac{-c_{1,i+1}}{c_{2,i+1} - c_{1,i+1}} \end{bmatrix} \cC_*(E_{i, R}) \eas and \bas \cC_*'(E_{i,L}) = \begin{bmatrix} \frac{c_{2,i-1}}{c_{2,i-1} - c_{1,i-1}} & \frac{-c_{1,i-1}}{c_{2,i-1} - c_{1,i-1}} \\ 0 & 1  \end{bmatrix} \cC_*(E_{i, L}) \;.\eas Under these changes of variables, the matrices $\cC'_*(E_{i,R})$ and $\cC'_*(E_{i, L})$ have the same form as $\cC_*(E_i)$. That is, they all have $0$s, $1$s and variables in the same positions.\footnote{The interested reader is invited to multiply out the expressions for $\cC_*'(E_{i,\bullet})$ and verify this for themselves!} Let $\cI(\cC_*'(E_{i,R}))$ and $\cI(\cC_*'(E_{i,L}))$ be the integrals $\cI(E_{i,R})$ and $\cI(E_{i,R})$ written in terms of the variables in the parametrization $\cC_*(E_i)$. By performing the appropriate changes of variables (as in Case 2), we obtain \bas \cI(\cC_*'(E_{i,R})) + \cI(\cC_*'(E_{i,L})) = \cI(\cC_*(E_i)) \;,\eas
as required.
\end{proof}

\begin{cor}
All spurious poles cancel out in the sum
\[\mathfrak{W}_{2,6} = \sum_{\substack{W \; admiss.\\ n\; verts, k\; props}} \cI(W)(\cZ_*).\]
\end{cor}
\begin{proof}
This follows directly from Theorem \ref{boundarycancelthm}.
\end{proof}

\begin{rmk}
There are 6 cells (each lying on the boundary of a cell of the form $\Sigma(V_i)$), that do not appear in the calculations of Theorem \ref{boundarycancelthm} at all. These are precisely the 5-dimensional cells which cannot be realized as cells associated to boundary diagrams of Wilson loop diagrams.
\end{rmk}

\bibliographystyle{abbrev}

\end{document}